\documentclass{article}
\usepackage[left=2.65cm,right=2.65cm,top=2.7cm,bottom=2.7cm]{geometry}
\usepackage{amsmath}
\usepackage{amsthm}
\usepackage{amscd}
\usepackage{amssymb}
\usepackage{amsfonts}
\usepackage{graphics}
\usepackage[dvipsnames]{xcolor}
\usepackage{cancel}
\usepackage{comment}
\usepackage{enumitem}
\usepackage{appendix}

\usepackage{hyperref}
\usepackage{mathtools}
\mathtoolsset{showonlyrefs}

\usepackage{tikz}
\usetikzlibrary{arrows.meta, positioning, calc}
\usepackage{pgfplots}
\pgfplotsset{compat=1.18}

\hypersetup{
	colorlinks   = true, 
	urlcolor     = blue, 
	linkcolor    = blue, 
	citecolor   = red 
}

\newcommand{\boC}{\mathcal{C}}
\newcommand{\boK}{\mathcal{K}}
\newcommand{\boI}{\mathcal{I}}

\newcommand{\R}{\mathbb{R}}
\newcommand{\boW}{\mathcal{W}}
\newcommand{\boN}{\mathcal{N}}
\newcommand{\boF}{\mathcal{F}}
\newcommand{\boE}{\mathcal{E}}
\newcommand{\boH}{\mathcal{H}}
\newcommand{\boA}{\mathcal{A}}
\newcommand{\boG}{\mathcal{G}}

\newcommand{\dq}{\dot{q}}
\newcommand{\dpp}{\dot{p}}

\newcommand{\Hilb}{{\textnormal H}}
\newcommand{\Ebanach}{{\textnormal E}}

\newcommand{\minq}{{\bf q}}
\newcommand{\ball}{B_\boK}
\newcommand{\X}{{\textnormal X}}
\newcommand{\PS}{(\textnormal{PS})^{*}_\tau}

\DeclareMathOperator{\dv}{\hbox{div}}
\newcommand{\taumax}{T}

\newtheorem{theorem}{Theorem}[section]
\newtheorem{lemma}{Lemma}[section]
\newtheorem{proposition}{Proposition}[section]
\newtheorem{corollary}{Corollary}[section]
\theoremstyle{definition}
\newtheorem{remark}{Remark}[section]
\newtheorem{example}{Example}[section]
\newtheorem{definition}{Definition}[section]

\title{A Local Linking Theorem for
Relativistic Action Functionals}
\author{ Manuel Garz\'on\footnote{Instituto de Matemáticas IMUS, Universidad de Sevilla, 41012 Sevilla, Spain. E-mail: {\tt mgarzon2@us.es}}, Salvador L\'opez-Mart\'inez\footnote{Departamento de Matem\'aticas, Universidad Aut\'onoma de Madrid, Ciudad Universitaria de Cantoblanco, 28049, Madrid, Spain. 
		E-mail: {\tt salvador.lopez@uam.es}}}
\date{}

\begin{document}
	
	\maketitle

	\begin{abstract}
We establish an analogue of the Brezis--Nirenberg local linking theorem for a class of Szulkin-type functionals arising from relativistic action principles. In this framework, compactness of Palais--Smale sequences is formulated with respect to a topology induced by the effective domain of the functional, replacing the classical strong Palais--Smale condition. The proof combines the original construction of the min-max geometry, based on a negative gradient flow, with the Ekeland--Lasry regularization. The main difficulty is that the regularized functional is naturally associated with the strong topology of the underlying functional space, whereas compactness for the original functional is formulated in the topology induced by the effective domain. We overcome this obstacle through a new perturbative construction that recovers the required min-max structure. We apply our abstract multiplicity result to two representative relativistic models: the Lorentz force equation, describing the dynamics of a charged particle in an electromagnetic field, and the Dirichlet problem for the prescribed mean curvature operator in Minkowski space. As a consequence, under natural assumptions, each problem admits at least two non-constant solutions.
\end{abstract}

\section{Introduction}
The classical result \cite[Theorem 4]{Bre-Nir}  by Brezis and Nirenberg asserts the following: \textsl{Let $X$ be a Banach space split as $X=X_1\oplus X_2$, with dim$X_1<\infty$. Let $\boI\in\mathcal C^1(\textnormal{X};\R)$ be bounded from below, with $\inf_{\textnormal{X}}\boI<0$, and satisfying the strong Palais--Smale condition. Assume that
\begin{equation}
\label{hyp:SplittingBrezis}
\left\{\begin{array}{ll}
\boI(q)\leq0,&\hbox{for all }q\in X_1\ \hbox{with }\|q\|\in[0,r],\\
\boI(q)\geq 0,&\hbox{for all }q\in X_2\ \hbox{with }\|q\|\in[0,r],
\end{array}\right.
\end{equation}
for some $r>0$. Then $\boI$ has at least two nonzero critical points.}

In this paper, we establish an analogue of this theorem for a subclass of Szulkin-type functionals, which arise naturally in variational formulations of relativistic action principles. Formally, a prototypical relativistic action functional is 
\begin{equation}\label{eq:realtivistic_functional}
\boI(q)=\int_\Omega\left(1-\sqrt{1-|\nabla q|^2} + F(x,q,\nabla q)\right) dx,\qquad q\in\Hilb,
\end{equation}
where $\Omega\subset\R^N$ is a fixed bounded domain, $F$ is a given function and $\Hilb$ denotes a Sobolev space of functions defined in $\Omega$. In contrast with \cite[Theorem 4]{Bre-Nir}, functionals of the form \eqref{eq:realtivistic_functional} are non-smooth and, in general, do not satisfy a strong Palais--Smale  condition. Instead, compactness is formulated with respect to a topology $\tau$ naturally induced by the effective domain $\boK\subseteq\Hilb$ (see \cite{ABT,GLM}).  We apply our abstract result to the Lorentz force equation with vanishing scalar potential, under physically relevant assumptions. We also analyze the Dirichlet problem for the prescribed mean curvature operator in Minkowski space. In both situations, we obtain the existence of at least two non-constant solutions. More details of these applications are given below. 

The local linking condition \eqref{hyp:SplittingBrezis} was first introduced in \cite{Li-Liu, Li-Liu2} in a slightly more restrictive setting than in \cite{Bre-Nir}. Since then, a large body of work has been devoted to the analysis and applications of this condition in the context of critical point theory for $\mathcal C^1$ functionals. We refer the reader to \cite{FPV,Li-Willem,MBR,Perera} for some representative references. Based on Chang's variational theory \cite{Chang}, there exist extensions of \cite[Theorem 4]{Bre-Nir} to the non-smooth framework of locally Lipschitz functionals on reflexive Banach spaces. More precisely, \cite[Theorem 2.3]{Wu} assumes strong convergence of the corresponding Palais--Smale sequences, while \cite[Theorem 8]{KKP} replaces the Palais--Smale condition by the non-smooth Cerami condition, a weaker compactness assumption. Despite their relevance, none of these result applies to the relativistic problems motivating the present work, as they rely on compactness conditions in the strong topology. Indeed, functionals like \eqref{eq:realtivistic_functional} fall outside the locally Lipschitz framework of \cite{Wu,KKP}.

To the best of our knowledge, \cite[Theorem 3.1]{LMM} is the only existing extension of \cite[Theorem 4]{Bre-Nir} within the framework of Szulkin-type functionals. Such a result relies on a positivity assumption near the boundary of $\boK$ (see $f_6$ in \cite{LMM}), as well as on the strong Palais--Smale condition. Even assuming a compactness condition in $\tau$, a positivity assumption near the boundary of $\boK$ is not compatible with our setting (see Appendix~\ref{Appendix-Boundary}) Therefore a new approach is required for our purposes.

The main result of the paper is Theorem~\ref{MainTheorem}, which extends \cite[Theorem 4]{Bre-Nir} to the framework described above. Our assumptions retain the existence of a negative global minimum and a local linking geometry, while replacing the strong Palais--Smale condition by a suitable compactness condition in $\tau$, namely $\PS$ (see Definition~\ref{def:Funct-PS}). The proof combines ideas from the classical argument of Brezis and Nirenberg, via a negative gradient flow, with the regularization procedure introduced by Ekeland and Lasry in \cite{Eke-Las}, which requires the Hilbert space setting. A central difficulty is that the regularized functionals naturally lead to a strong topology, whereas compactness is only available in $\tau$. Our method develops a new construction which exploits simultaneously the regularity properties of the regularized functional and the $\tau$-compactness carried by $\boI$. This yields a suitable min-max geometry which provides the existence of a second nonzero critical point through the variational principle \cite[Theorem 1]{ABT} by Arcoya, Bereanu and Torres. Section~\ref{Sect:Scheme} outlines the proof and highlights its main difficulties.

The first application of Theorem~\ref{MainTheorem} concerns the multiplicity of nontrivial periodic solutions of the Lorentz force equation, which models the dynamics of a charged particle in an electromagnetic field. Critical point methods for such dynamics were recently initiated in \cite{ABT}, where the Poincaré action functional \cite{Po2} is naturally defined on the space of periodic Lipschitz functions (with fixed period). Since then, a number of contributions have been devoted to the existence and multiplicity of periodic solutions of the Lorentz force equation in different regimes. The case of continuous electromagnetic fields is studied in \cite{ABT2,Ber2,Ber-Pir}, whereas \cite{AS,Ber1,BosDamPap,GLM,GT} deal with models allowing the presence of singularities in the field.

The electromagnetic field is usually described in terms of scalar and vector potentials. This formulation is particularly convenient from a mathematical viewpoint due to the gauge invariance of Maxwell's equations, see \cite[Section 3]{GM} for details. In this context, models with vanishing scalar potential are fully natural, as they describe fields generated by configurations of electric currents. Using non-variational techniques, different aspects of the dynamics have been studied in both Newtonian and relativistic settings, see, for instance, \cite{ALP,GT2,GP1,GP2}. From a variational perspective, a novel approach was introduced in \cite{GLM}, by considering fields mainly described by the vector potential rather than by the scalar potential. More precisely, \cite[Theorem 2.1]{GLM} establishes the existence of a global minimizer for a broad class of vector potentials $\boA\subset\mathcal{C}^1(\R^4;\R^3)$ with vanishing scalar potential. For every $A\in\boA$, the minimizer is a non-constant periodic solution of the Lorentz force equation
\begin{equation}\label{eq:LFE-intro}
\dfrac{d}{dt}\left(\frac{\dq(t)}{\sqrt{1-|\dq(t)|^2}}\right)=E(t,q(t)) + \dq(t)\times B(t,q(t)),\quad\hbox{with}\quad E=-\partial_t A,\ B=\nabla\times A.
\end{equation}
Here $q(t)$ denotes the position of the particle, $\dq(t)$ its velocity, and $(E,B)$ the electromagnetic field. Prior to that work, there appeared to be a technical obstruction to the variational treatment of \eqref{eq:LFE-intro}.

In Theorem~\ref{MainTheorem:LFE} we provide a second non-constant periodic solution (with same period) for a subclass $\boA_0\subset\boA$, characterized by the presence of isolated equilibria of the electromagnetic field, thereby extending \cite[Theorem 2.1]{GLM} in this regime. The proof of Theorem~\ref{MainTheorem:LFE} relies on the application of Theorem~\ref{MainTheorem}. The details are found in Section~\ref{Sect:LFE}. It is worth mentioning that the Poincaré action functional has to be formulated on the Sobolev space of periodic functions $H^1$ (with fixed period), which requires extending the variational framework developed in \cite{ABT,GLM} from the space of Lipschitz functions to $H^1$. Such reformulation has also been employed recently in \cite{Ber-Pir} in order to apply the Ekeland--Lasry regularization \cite{Eke-Las}, but with different purposes. As already observed by Szulkin, this method fits naturally within the framework of \cite{Szu}.

The second application concerns the Dirichlet problem
\begin{equation}\label{eq:PDE-Minkowski}
  \dv\!\left(\frac{\nabla q}{\sqrt{1-|\nabla q|^2}}\right)=f(x,q),
  \quad x\in\Omega,
  \qquad
  q|_{\partial\Omega}=0,
\end{equation}
where $\Omega\subset\mathbb{R}^N$ ($N\geq 1$) is a bounded domain and $f\colon\Omega\times\mathbb{R}\to\mathbb{R}$ is a given function. From a geometric viewpoint, the graphs of the solutions~$q$
to~\eqref{eq:PDE-Minkowski} are hypersurfaces in the Minkowski space $\mathbb{L}^{N+1}=\{(x,t)\in\R^N\times \R\}$, with the flat metric $\sum_{i=1}^N(dx_i)^2-(dt)^2$, whose
boundary coincides with $\partial\Omega\subset\mathbb{R}^N$ and whose mean
curvature at $(x,q(x))$ equals $f(x,q(x))/N$. In this context, the (formally defined) operator $q\mapsto\dv ((1-|\nabla q|^2)^{-1/2}\nabla q)$ is referred to as the \emph{mean curvature operator}. Problem \eqref{eq:PDE-Minkowski} has also a physical interpretation through the Born--Infeld model of Electromagnetism, especially when $\Omega=\R^N$ and the boundary condition is replaced with a vanishing condition at infinity, see \cite{BonDavPom,ByeIkoMalMar1}.

In the seminal paper \cite{BS}, the problem \eqref{eq:PDE-Minkowski} is studied by minimizing the associated action functional with $f\in\boC(\Omega\times\R)$. A key step in the analysis is to prove that the minimizer is indeed a solution of \eqref{eq:PDE-Minkowski}. A similar difficulty arises in the more recent work \cite{ByeIkoMalMar2}, where the low-regularity case $f(x,s)=\rho(x)$, with $\rho$ a finite signed Borel measure, is considered. It is also observed in \cite{BS} that, if $s\mapsto f(x,s)$ is non-decreasing, the action functional is strictly convex, yielding uniqueness of the minimizer. In situations where this convexity property fails, a second solution to \eqref{eq:PDE-Minkowski} is obtained in \cite{BerJebMaw} by exploiting the non-smooth mountain pass theorem \cite[Theorem 3.2]{Szu}. Multiplicity of solutions has also been established by non-variational techniques. Namely, \cite{BerJebTor} considers \eqref{eq:PDE-Minkowski} in the radial case with $N\ge2$ and $f(x,s)=-\lambda\mu(|x|)s^q$, where $\mu$ is continuous and positive, $q>1$, and $\lambda>0$ is a parameter. By means of degree theory, the authors show that at least two solutions exist when $\lambda$ is large enough.

Our contribution regarding the multiplicity of solutions to \eqref{eq:PDE-Minkowski} is Theorem \ref{MainTheorem-Dirichlet}. This result shows that Theorem \ref{MainTheorem} yields the existence of at least two nontrivial solutions for nonlinearities of the form $f(x,s)=a(x)s-g(s)$, with $a$ within a suitable class generating negative Laplacian eigenvalues, and $g\in\boC^1$ satisfying $g(0)=g'(0)=0$. To the best of our knowledge, this is the first application of a local linking theorem to a Dirichlet problem driven by the mean curvature operator in Minkowski space. The assumptions of Theorem \ref{MainTheorem} are verified in Section \ref{Section:Minkowski}.

The paper is organized as follows. Section~\ref{Sect:FunctionalFramework} contains the abstract multiplicity theorem and its proof. To guide the reader, the main ideas of the proof and the principal difficulties are discussed in Section~\ref{Sect:Scheme}. Sections~\ref{Sect:LFE} and~\ref{Section:Minkowski} are devoted to the Lorentz force equation and the prescribed mean curvature equation in Minkowski space, respectively. Appendix~\ref{appendix-lemma} contains a technical result concerning the topology $\tau$ and the effective domain in the prototypical relativistic framework. Appendix~\ref{Appendix-Boundary} shows that the positivity assumption on the boundary $\partial\boK$ of the action functional $\boI$ is not compatible with our setting. For each application, we construct a one-parameter family of functionals for which this assumption fails.

\section{The non-smooth local linking theorem }\label{Sect:FunctionalFramework}
Let $\Hilb$ be a Hilbert space with a direct sum decomposition $\Hilb=\Hilb_1\oplus\Hilb_2$, with $\text{dim } \Hilb_1<\infty$. 
Throughout this paper, we denote the norm and inner product in $\Hilb$ by $\|\cdot\|_\Hilb$ and $\langle\cdot,\cdot\rangle_\Hilb$ respectively. Let also $(\Ebanach,\|\cdot\|_\Ebanach)$ be a Banach space in which $\Hilb$ is continuously embedded, i.e. there exists a constant $c_0>0$ such that  
    \begin{equation}\label{eq:embedding}
\|q\|_{\textnormal E}\leq c_0 \|q\|_\Hilb,\quad\hbox{for all }q\in\Hilb.
\end{equation}
   
The following definition gathers the essential properties of the domains of relativistic operators, highlighting the crucial role of a specific topology in these domains.
\begin{definition}\label{def:domain-K}
A convex, closed in $\Hilb$ set $\boK$ is said to be a $\tau$-domain if there exists a topology $\tau$ on $\boK$ satisfying the following two properties:
\begin{enumerate}
\item[(i)] For all $\eta>0$ and $q\in\boK$, the balls
\begin{equation}\label{def:ball}
\ball(q,\eta):=\{p\in\boK:\ \|p-q\|_{\textnormal E}\leq \eta\}
\end{equation}
are compact with respect to $\tau$.
\item[(ii)] The topology $\tau$ is stronger than the topology induced by the norm in ${\textnormal E}$, i.e.,   \begin{equation}\label{hyp:tau_stronger}
\|q_n-q\|_{\textnormal E}\to 0,\text{ as }n\to\infty,\text{ for every }\{q_n\}\subset\boK,\ q\in\boK\text{ such that }q_n\rightarrow^{\tau}q, 
\end{equation}
where $\rightarrow^{\tau}$ denotes convergence in $\tau$.
\end{enumerate}
\end{definition}
\begin{remark}\label{Remarks:Prototypical-Examples}
The prototypical examples are
\begin{equation}\label{eq:example_spaces}
\Hilb=H^1(\Omega),\quad {\textnormal E}=L^p(\Omega),\quad \boK=\{q\in H^1(\Omega): \|\nabla q\|_{\infty}\leq 1\},
\end{equation}
for some bounded domain  $\Omega\subset\R^N$ with smooth boundary, and for any $p>1$ Sobolev-subcritical. Specifically, one may take
\[p=\frac{2N}{N-2}\text{ if }N\geq 3,\quad p<\infty\text{ if }N=2,\quad\text{and}\quad p=\infty \text{ if }N=1.\]
The spaces in \eqref{eq:example_spaces} could be scalar or vectorial, and could as well be complemented with Dirichlet or periodic boundary conditions, for instance. In this setting, it is well known that \eqref{eq:embedding} holds. Moreover, Lemma~\ref{appendix-lemma:tau-domain} in the Appendix shows that $\boK$ is a $\tau$-domain for the topology induced by the norm
\begin{equation}\label{eq:topology-K}
\|q\|_\boK = \|q\|_{\Ebanach} + \|\nabla q\|_{w^*},
\end{equation}
where $\|\cdot\|_{w^*}$ denotes the norm associated with the weak$^*$ topology $\sigma(L^\infty, L^1)$ (see  \cite[Theorem 3.28]{Bre}).
\end{remark}

We next introduce the functionals under consideration, which form a subclass of the Szulkin-type functionals introduced in \cite{Szu} and capture the main features of relativistic action principles. Specifically, let $\boK\subset \Hilb$ be a fixed $\tau$-domain, and let $\boI: \Hilb\to (-\infty,\infty]$ admit a decomposition of the form $\boI=\Psi+\boF$, satisfying the following assumptions:
	\begin{enumerate}[label=(H\arabic*)]   
		\item $\Psi:\Hilb\to (-\infty,\infty]$ is a proper and convex functional with domain $\boK=\{q\in\Hilb:\Psi(q)<\infty\}$. Moreover, $\Psi|_{\boK}$ is continuous with respect to $\|\cdot\|_\Hilb$, and lower semicontinuous with respect to $\tau$. \label{H1:nonsmooth}
		\item $\boF:\Hilb\to\R$ is of class $\mathcal{C}^1$. Moreover, $\boF|_\boK$ is locally uniformly continuous in ${\|\cdot\|_\textnormal E}$, namely, given $\varepsilon>0$ and $q_0\in\boK$, there exists $\delta:=\delta(q_0,\varepsilon)>0$ such that\begin{equation}\label{hyp:F-dualE}
		|\boF(q)-\boF(p)|<\varepsilon,\quad\hbox{for all }q,p\in B_\boK(q_0,\delta).\end{equation}\label{H2:smooth} 
\end{enumerate}
\begin{remark}
Since $\boI(q)=\infty$ for every $q\in\Hilb\setminus\boK$, the functional is trivially lower semicontinuous in $\|\cdot\|_\Hilb$ on the whole space. Moreover, let us emphasize that condition \eqref{hyp:F-dualE} is not standard in the context of Szulkin's functionals. This property will play a crucial role in the proof of Lemma~\ref{lemma-alpha}.
\end{remark}

We use the following notation for the space of functionals with the above decomposition:
    \begin{equation}
        \X(\boK)=\{\boI=\Psi+\boF:\Hilb\to (-\infty,\infty]\text{ satisfying  \ref{H1:nonsmooth} and \ref{H2:smooth}}\}.    \end{equation}
In the setting of $\X(\boK)$, critical points are defined as follows:
\begin{definition}\label{def:CriticalPoint}
A point $q \in \boK$ is called a critical point of $\boI \in \X(\boK)$ if it satisfies
\begin{equation}\label{CriticalPoint}
    \Psi(p) - \Psi(q) + \boF'(q)[p - q] \geq 0, \quad \text{for all } p \in \boK.
\end{equation}
The set of critical points of $\boI$ will be denoted by
\begin{equation}\label{eq:critical_point_set}
     \Sigma(\boI):=\{q\in\boK:\ q \hbox{ satisfies \eqref{CriticalPoint}}\}.
\end{equation}
\end{definition}

\begin{remark}
Notice that  condition \eqref{CriticalPoint} holds trivially when $p \in \Hilb \setminus \boK$. Therefore, restricting to the case $p \in \boK$ entails no loss of generality.
\end{remark}

We recall that every local minimizer of $\boI$ is a critical point (see \cite[Proposition 1.1]{Szu}). Complementarily, critical points other than minimizers are typically obtained as limits of Palais--Smale sequences. These sequences are also well-defined in the non-smooth functional framework introduced in \cite{Szu}:

\begin{definition}\label{def:PalaisSmale}
A sequence $\{q_n\} \subset \boK$ is called a Palais--Smale sequence for $\boI \in \X(\boK)$ at level $c \in \R$ if
\begin{equation}
    \lim_{n \to \infty} \mathcal{I}(q_n) = c,
\end{equation}
and there exists a sequence $\{\epsilon_n\} \subset (0,\infty)$, with $\epsilon_n \to 0$ as $n\to\infty$, such that
\begin{equation}\label{eq:PalaisSmale}
    \Psi(p) - \Psi(q_n) + \boF'(q_n)[p - q_n] \geq -\epsilon_n\,\|p - q_n\|_{\Hilb}, \quad \text{for all } p \in \boK.
\end{equation}
\end{definition}

A standard assumption in variational theory is the compactness of the Palais--Smale sequences in the strong topology of $\Hilb$. Under this hypothesis, \cite[Proposition 1.4]{Szu} implies that the limit of a sequence satisfying \eqref{eq:PalaisSmale} is a critical point at the corresponding level. However, establishing such a condition is often highly nontrivial, even in smooth settings, and in many situations it cannot be expected to hold. Nevertheless, in certain relativistic variational frameworks, the geometry of the functional naturally gives rise to an alternative notion of compactness, which in our setting corresponds to convergence with respect to the topology $\tau$ (see \cite{ABT,GLM}). Motivated by the application presented in Section~\ref{Sect:LFE} (see also Remark~\ref{Remark:PS-0}), the following definition adapts this notion to our abstract framework:

\begin{definition}\label{def:Funct-PS}
The functional $\boI \in \X(\boK)$ satisfies $\PS$ if every Palais--Smale sequence $\{q_n\} \subset \boK$ at level $c\neq0$ admits a subsequence converging in $\tau$ to a critical point $q \in \boK$ with $\boI(q)=c$.
\end{definition}

We now present the main result of this work, which extends \cite[Theorem 4]{Bre-Nir} to the class $\X(\boK)$.

\begin{theorem}\label{MainTheorem}
Let $\boI \in \X(\boK)$ be bounded from below,  with $\inf_{\Hilb} \boI < 0$, and satisfy $\PS$. Assume that
\begin{equation}\label{hyp:Splitting}
\left\{\begin{array}{ll}
\boI(q)\leq0,&\hbox{for all }q\in\Hilb_1\ \hbox{with }\|q\|_\Hilb\in(0,r],\\
\boI(q)\geq g(\|q\|_\Hilb),&\hbox{for all }q\in\Hilb_2\ \hbox{with }\|q\|_\Hilb\in(0,r],
\end{array}\right.
\end{equation}
for some $r>0$, and some function $g : (0,r] \to (0,\infty)$. If there exists $\mu \geq 0$ such that 
\begin{equation}\label{hyp:convex}
		\hbox{$\boI + \mu\|\cdot\|_\Hilb^2$}\text{ is convex},
\end{equation}
then $\boI$ admits at least two nonzero critical points.
\end{theorem}

\begin{remark}\label{rem:I(0)}
The first inequality in \eqref{hyp:Splitting} implies that $B_{r}:=\{y\in\Hilb_1:\, \|y\|_\Hilb\leq r\}\subset\boK$, hence the map $\boI|_{B_{r}}:B_r\to\R$ is continuous. Combining this with \eqref{hyp:Splitting}, we obtain $\boI(0)=0$. 
\end{remark}

The proof of Theorem~\ref{MainTheorem} is given at the end of Section~\ref{sect:min-max} and follows the strategy in \cite[Theorem 4]{Bre-Nir}, adapted here to the setting of $\X(\boK)$. A formal scheme of the proof is presented in Section~\ref{Sect:Scheme},  outlining the main ideas and difficulties. In Section~\ref{Sect:tau-levelsets} we establish suitable control of the level sets of $\boI$ with respect to $\tau$, where $\PS$ plays a crucial role. Section~\ref{Sect:eke-las} applies the Ekeland–Lasry regularization technique from \cite{Eke-Las} to the class $\X(\boK)$. Section~\ref{Sect:Negative-Flow} exploits the key tools provided in the previous sections, combining them with a negative gradient flow argument. This yields a suitable min-max geometry in the setting of $\X(\boK)$. Finally, Section~\ref{sect:min-max} develops a mountain pass argument to complete the proof of Theorem~\ref{MainTheorem}.

\subsection{Formal scheme of the proof}\label{Sect:Scheme}

To illustrate the approach of \cite[Theorem 4]{Bre-Nir}, let $\boI:\Hilb\to \R$ be of class $\boC^1$ with globally Lipschitz derivative $\boI'$, satisfying the classical strong Palais--Smale condition, that is, any sequence $\{q_n\}\subset \Hilb$ such that $\boI(q_n)\to c\in\R$ and $\|\boI'(q_n)\|_\Hilb\to 0$ admits a subsequence that converges strongly in $\Hilb$. Assume in addition that $\boI$ admits a unique global minimizer $\minq\in\Hilb$ such that $\boI(\minq)<0=\boI(0)$. Observe that, in order to establish multiplicity for $\boI$, it suffices to consider the case in which $\minq$ is the unique critical point of $\boI$ at a negative level. In this situation, the proof of \cite[Theorem 4]{Bre-Nir} combines a condition \eqref{hyp:SplittingBrezis} with the existence of a negative global minimum. These ingredients yield a  suitable min-max geometry, from which the Palais--Smale condition yields a critical point distinct from both $0$ and $\minq$.

To obtain the min-max geometry, let us fix a vector $v\in \Hilb_2$, with $\|v\|_\Hilb=r<\|\minq\|_\Hilb$, and define the compact set
\begin{equation}\label{def:Set-N}
\boN:=\{q\in\Hilb:\ q=sv+y,\ \hbox{ with }y\in\Hilb_1,\ s\in[0,1]\ \hbox{and } \|q\|_\Hilb\leq r\},
\end{equation}
whose boundary is
\begin{equation}\label{def:BoundarySet-N}
\partial\boN=\boN_1\cup\boN_2,\quad\text{where }\boN_1=\{y\in\Hilb_1:\|y\|_\Hilb\in[0,r]\},\quad\boN_2=\{q\in\boN:\ \|q\|_\Hilb= r\}.
\end{equation}
The goal is to construct a continuous map $\gamma^*:\partial\boN\to\R$  satisfying the following three properties:
\begin{equation}\label{eq:formal_properties}
     \gamma^*|_{\boN_1}=\text{Identity};
    \qquad  \inf_{q\in\boN_2}\|\gamma^*(q)\|>0;\qquad \sup_{q\in\partial\boN}\boI(\gamma^*(q))= 0.
\end{equation}
In that case, \cite[Lemma~3]{Bre-Nir}, together with conditions \eqref{hyp:SplittingBrezis} and \eqref{eq:formal_properties}, yields
\begin{equation*}
\sup_{q\in\partial\boN}\boI(\gamma^*(q))=0<\inf_{\gamma\in\Gamma}\sup_{q\in\boN}\boI(\gamma(q)),
\end{equation*}
where $
    \Gamma:=\{\gamma:\boN\to\Hilb\ \hbox{ continuous and }\gamma|_{\partial\boN}=\gamma^*\}.$  To clarify, such map $\gamma^*$ is refereed to as $p^*$ in \cite{Bre-Nir}.
    
    Let us focus on constructing $\gamma^*$. Since we are interested in nonzero critical points different from $\minq$, one may assume without loss of generality that
\[\boI'(y)\not=0,\quad\text{for every }y\in\Hilb_1\text{ with }\|y\|_\Hilb=r.\]
This allows us to introduce the negative gradient flow starting at $y\in \Hilb_1$ with $\|y\|_\Hilb=r$, that is, the solution to the Cauchy problem
\begin{equation}
    \frac{dx}{dt}=-\frac{\boI'(x(t))}{\|\boI'(x(t))\|_\Hilb^2},\quad x(0)=y.
\end{equation}
Direct computations show that
\[\boI(x(t;y))=\boI(y)-t,\quad\text{for all }t\in [0,\taumax(y)),\]
where $\taumax(y)\in (0,\infty]$ is the maximal time of existence. Since $\boI$ is bounded from below, the above relation yields the upper bound $T(y)\leq-\boI(\minq)$, for all $y$. In particular, the blow-up of $\frac{dx}{dt}(t;y)$ occurs in finite time, since $\boI'$ vanishes as $t\to\taumax(y)$ and the right-hand side of the ODE becomes singular. More precisely, assuming that there are no critical points at negative levels other than $\minq$ (again without loss of generality), the strong Palais--Smale condition yields the convergence
\begin{equation*}
\|x(t;y)-\minq\|_\Hilb\to 0,\quad\text{as }t\to \taumax(y),
\end{equation*}
which implies the continuity of $x(\cdot;y)$ at $\taumax(y)$. Moreover, since $\boI(y)\leq 0$ by \eqref{hyp:SplittingBrezis}, and the function $t\mapsto\boI(x(t;y))$ is decreasing, it follows that
\[\boI(x(t;y))< 0,\quad\text{for all }t\in (0,\taumax(y)).\]
Therefore, assuming that $y\mapsto\taumax(y)$ and $(t,y)\mapsto x(t;y)$ are both continuous, one may formally set
\[\gamma^*(q)=q,\text{ if }q\in\boN_1,\quad \gamma^*(q)=x(s\taumax(y);y),\text{ if } q=sv+y\in\boN_2,\]
which satisfies \eqref{eq:formal_properties}.

To handle the possible lack of continuity with respect to the initial datum, the idea is to take the time $t(y)$ at which $x$ crosses the level set $\{\boI(q)=\boI(\minq)+\delta\}$, for $\delta>0$ small enough, and then connect $x(t(y);y)$ to the global minimizer $\minq$ by a straight segment. Continuity then follows by standard arguments. The trade-off is that one must verify the third property along these segments, at least for $\delta>0$ small enough. To this end, as noted in \cite[Proposition 2]{Bre-Nir}, the strong Palais--Smale condition combined with the Ekeland variational principle implies that every minimizing sequence of $\boI$ admits a convergent subsequence. Hence, the uniqueness of the minimizer guarantees that the level sets near $\minq$ are contained in an arbitrarily small ball centered at $\minq$. More precisely, for every $\varepsilon>0$, there exists $\delta>0$ such that
\begin{equation}
    \{q\in\Hilb:\ \boI(q)=\boI(\minq)+\delta\}\subset\{q\in\Hilb: \|q-\minq\|_\Hilb<\varepsilon\}.
\end{equation}
Since the straight segments connect points on this level set to the minimizer $\minq$, every segment remains entirely within $\{q\in\Hilb:\|q-\minq\|_\Hilb<\varepsilon\}$, by convexity and the previous inclusion. Finally, as $\boI(\minq)<0$, continuity of $\boI$ ensures that $\boI<0$ throughout the ball for $\varepsilon>0$ sufficiently small, so the third property in \eqref{eq:formal_properties} still holds along the segments.

We now wish to adapt this approach to the class $\X(\boK)$. The difficulties we encounter in building the min-max structure are essentially twofold. First, the action functional $\boI$ is not of class $\boC^1$ in $\Hilb$ and, in turn, the Cauchy problem may not be well-posed for $y\in\Hilb_1$. Even if the flow $x(t;y)$ may be well-defined in the interior of $\boK\cap\Hilb_1$, it may cease to exist before reaching $\minq$, in case it attains the boundary $\partial\boK$. This issue may be circumvented under assumption $(f_6)$ in \cite{LMM}, which imposes a positivity condition on $\boI$ near $\partial\boK$. In particular, this prevents the flow from approaching $\partial\boK$. 
 However, one can readily construct examples of the form \eqref{eq:realtivistic_functional} for which $\boI$ attains negative values near $\partial\boK$ (see Appendix~\ref{Appendix-Boundary}). Hence, this hypothesis cannot be assumed in our setting. We overcome this difficulty by replacing $\boI$ in the Cauchy problem with the regularized functional $\boI_\varepsilon$ in the sense of \cite{Eke-Las}.  Among the advantages of this approach, we emphasize that $\boI_\varepsilon$ is smooth, is bounded above by $\boI$, and, remarkably, has the same critical points as $\boI$. Moreover, there is a close relationship between the Palais--Smale sequences of the two functionals. The price to pay is that the regularization requires a Hilbert space setting, as well as assumption \eqref{hyp:convex}, see Section~\ref{Sect:eke-las} for details. As the negative gradient flow is now associated to $\boI_\varepsilon$, we denote it by $x_\varepsilon$.

The second difficulty arises in connection with the Palais--Smale condition. The regularization naturally imposes the strong topology of $\Hilb$ as the appropriate one for working with Palais--Smale sequences for $\boI_\varepsilon$ (see Proposition~\ref{pro:regularizing} and Remark~\ref{remark:difficulties}). Nevertheless, as noted before, a strong Palais--Smale condition in $\Hilb$ is not expected to hold in general for relativistic functionals. Consequently, continuity in $t$ of the negative gradient flow cannot be established directly as in the classical case. 

The problem is further complicated by the need for continuity in $y$ as well, which requires precise control on the level sets. To illustrate this difficulty, consider the best-case scenario in which $x(t;y)$ does not reach $\partial\boK$ for any $t\in(0,T(y))$, so that no regularization is needed. Arguing as in \cite{Bre-Nir}, the alternative Palais--Smale condition $\PS$ implies that \[x(t;y)\to^\tau \minq,\quad\text{as }t\to T(y).\] When attempting to gain continuity of $\gamma^*$, the standard argument used to verify the third property in \eqref{eq:formal_properties} breaks down, since $\boI$ is only lower semicontinuous in the $\tau$-topology rather than continuous. In particular, the sign of $\boI$ in a $\tau$-neighborhood of $\minq$ may not be preserved. 

The key ingredient in the construction of the map $\gamma^*$ satisfying
\eqref{eq:formal_properties} is that, instead of working directly with $x_\varepsilon$, we consider the perturbation
\[\omega_\varepsilon(x_\varepsilon)
:=
x_\varepsilon-\frac{\varepsilon}{2}\boI'_\varepsilon(x_\varepsilon).
\]
We will show that the map $\omega_\varepsilon(x_\varepsilon)$ is continuous in $\Hilb$ and satisfies
\[
\boI(\omega_\varepsilon(x_\varepsilon))
=
\boI_\varepsilon(x_\varepsilon)
-
\frac{\varepsilon}{4}\|\boI'_\varepsilon(x_\varepsilon)\|_\Hilb^2.
\]
As a consequence, at the time $t_\varepsilon(y)$ at which $x_\varepsilon$ crosses the level set $\{\boI_\varepsilon(q)=\boI(\minq)+\delta\}$, one has
\begin{equation}
\omega_\varepsilon(x_\varepsilon(t_\varepsilon(y);y))
\in
\{q\in\boK:\boI(q)\le \boI(\minq)+\delta\}
\subset
\{q\in\boK:\|q-\minq\|_\Ebanach<\varepsilon\},
\end{equation}
where the last inclusion follows from  $\PS$, see Lemma \ref{lemma-LevelSets}. By virtue of the local continuity of $\boI$ with respect to the $\Ebanach$-norm established in Lemma~\ref{lemma-alpha}, we can verify \eqref{eq:formal_properties}. This argument is carried out in Proposition~\ref{Pro:Gamma} and
constitutes the core of this work.

We conclude by pointing again that the local linking condition \eqref{hyp:Splitting} is stronger than \eqref{hyp:SplittingBrezis}. Through the additional assumption \eqref{hyp:convex}, we require the functional to be strictly positive for every nonzero element in the balls of $\Hilb_2$ with small radius. This  is necessary in our setting, since the Palais–Smale sequences \eqref{eq:PalaisSmale} may be non-compact at level zero in the examples under consideration (see Remark~\ref{Remark:PS-0}). Moreover, it allows us to conclude in the applications from Section~\ref{Sect:LFE} and Section~\ref{Section:Minkowski} that the critical points of the functional correspond to non-constant solutions.

\subsection{Properties of \texorpdfstring{$\X(\boK)$}{}}\label{Sect:tau-levelsets} 

In this section we establish three key properties of the functional framework $\X(\boK)$. We begin showing a uniform local continuity with respect to $\tau$ for $\boI\in\X(\boK)$. The uniformity is achieved through the convexity and lower semicontinuity of $\Psi$, together with condition \eqref{hyp:F-dualE}. This continuity property will provide the required control of the level sets of $\boI$ near its minimum, which is essential in the proof of Theorem~\ref{MainTheorem}.

\begin{lemma}\label{lemma-alpha}
	Let $\boI\in\X(\boK)$. Then, for any $q\in\boK$ and any $\alpha>0$, there exists $\eta>0$ such that
	\begin{equation}\label{ineq:lemma1}
		\boI(tp+(1-t) q)<\boI(p)+\alpha,\quad\hbox{for all }(t,p)\in [0,1]\times\ball(q,\eta),
	\end{equation} 
	where the ball $\ball$ is defined in \eqref{def:ball}.
\end{lemma}

\begin{proof}
Given $q\in \boK$ and $\alpha>0$, we first claim that
\begin{equation}\label{Claim-proof}
		\Psi(q)<\Psi(p)+\dfrac{\alpha}{2},\quad\hbox{for all }p\in\ball(q,\eta_1),
	\end{equation}
for some $\eta_1>0$. Additionally, by \eqref{hyp:F-dualE} there exists $\eta_2>0$ such that
\begin{equation}\label{ineq:Fmin}
		\boF(tp+(1-t)q)\leq \boF(p)+\frac{\alpha}{2}, \quad\hbox{for all }p\in\ball(q,\eta_2).
	\end{equation}
Taking $\eta\in (0,\min\{\eta_1,\eta_2\})$, using the convexity of $\Psi$ together with the previous inequalities, we obtain
	\begin{align}
		\boI(tp+(1-t)q)&\leq \Psi(p)+(1-t)(\Psi(q)-\Psi(p)) + \boF(p)+\frac{\alpha}{2}<\boI(p)+\alpha,\quad\hbox{for all }p\in\ball(q,\eta).
	\end{align}
Therefore the lemma holds if the claim is true. Assume by contradiction that it is false. Then there exist $q\in\boK$, $\alpha>0$, and a sequence $\{q_n\}\subset\boK$, with $ q_n\in B_\boK(q,n^{-1})$, such that
\begin{equation}
    \Psi(q)\geq\Psi(q_n)+\dfrac{\alpha}{2},\quad\hbox{for all }n\in\mathbb{N}.
\end{equation}
By \eqref{def:ball}, there exists $q_0\in\boK$ such that $q_n\xrightarrow{\tau} q_0$, up to a subsequence. In particular, $q_0=q$ by \eqref{hyp:tau_stronger}. Passing to the limit and using the lower semicontinuity of $\Psi$ in $\tau$, we obtain the next contradiction:
\begin{equation}
    \Psi(q)\leq\liminf_{n\to\infty} \Psi(q_n)\leq\Psi(q)-\dfrac{\alpha}{2}<\Psi(q).
\end{equation}
Therefore, the claim is satisfied.
\end{proof}
The next result guarantees the existence of a global minimum for functionals in $\X(\boK)$ bounded from below under condition $\PS$. It can be seen as a version of \cite[Proposition 2]{Bre-Nir} adapted to our setting.

\begin{lemma}\label{lemma:abstract-minimizer}
	Let $\boI\in\X(\boK)$ be bounded from below, with $\inf_\Hilb\boI<0$, and satisfy $\PS$. Then, every minimizing sequence admits a subsequence that converges in $\Ebanach$ to a global minimizer of $\boI$.
\end{lemma}
\begin{proof}
Let $\{q_n\}\subset\boK$ be minimizing sequence. By Ekeland's variational principle, it is straightforward to show that  there exists a Palais--Smale sequence $\{p_n\}\subset\boK$ at level $\inf_\Hilb\boI$ (and thus a minimizing sequence),  satisfying, up to a subsequence, $\|q_n-p_n\|_\Hilb\to 0$ as $n\to\infty$ (see, e.g., \cite[Proposition 1.6 and Theorem~1.7]{Szu}).  Consequently, $\PS$ implies that there exists $q\in\boK$ such that $p_n\to^\tau q$ as $n\to\infty$ and $q$ is a critical point of $\boI$ at level $\inf_\Hilb\boI$. Recall that \eqref{hyp:tau_stronger} also implies that $\|p_n-q\|_\Ebanach\to 0$ as $n\to\infty$. Therefore, the lower semicontinuity of $\Psi$ with respect to $\tau$ and the continuity of $\boF$ with respect to $\Ebanach$ imply that $\boI(q)=\inf_\boH\boI$, so that $q$ is a global minimizer. Finally, \eqref{eq:embedding} implies that $\|q_n-q\|_\Ebanach\to 0$ as $n\to\infty$.
\end{proof}

\begin{remark}
    Observe that, in the proof of Lemma~\ref{lemma:abstract-minimizer}, the fact that the embedding property \eqref{eq:embedding} holds in the entire Hilbert space $\Hilb$ is essential. We point that, in some non-smooth setting, an embedding property is only required to hold in the domain of the functional, see \cite[Theorem 2.1]{Bos}.
\end{remark}

Using \cite[Proposition 2]{Bre-Nir} and assuming the uniqueness of the global minimizer, the authors in \cite{Bre-Nir} show that the sublevel sets at sufficiently small levels are contained in a ball (in $\|\cdot\|_\Hilb$) centered at the minimizer. The following result provides the natural analogue of this property in our setting, where the strong topology of $\Hilb$ is replaced by that of $\Ebanach$.

\begin{lemma}\label{lemma-LevelSets}
	Let $\boI\in\X(\boK)$ satisfy $\PS$ and have a unique global minimizer $\minq\in\Hilb$. Then, for all $\eta>0$, there exists $\delta>0$ such that
	\begin{equation}\label{eq:InclusionLevelsets}
		\{q\in\boK:\ \boI(q)\leq\boI(\minq)+\delta\}\subset B_\boK(\minq,\eta),
	\end{equation}
    where $\ball$ is defined in \eqref{def:ball}.
\end{lemma}
\begin{proof}
Assume that \eqref{eq:InclusionLevelsets} is false. Then there exist $\eta>0$ and a sequence $\{q_n\}\subset\boK$ such that
	\begin{equation}\label{eq:Claim-False}
		\|q_n-\minq\|_{\textnormal E}\geq\eta,\quad\hbox{and}\quad	\boI(q_n)\leq \boI(\minq)+\dfrac{1}{n},\quad\hbox{for all }n\in\mathbb{N}.
	\end{equation}
	In particular, $\{q_n\}$ is a minimizing sequence. Therefore, since the minimizer is unique, Lemma~\ref{lemma:abstract-minimizer} implies that
    $\|q_n-\minq\|_\Ebanach\to 0$ as $n\to\infty$. This is a contradiction.
\end{proof}

\subsection{The Ekeland--Lasry regularization}\label{Sect:eke-las}

For any $\boI\in\X(\boK)$ and $\varepsilon>0$, let us define the functional $\boI_\varepsilon:\Hilb\to\R$ as follows:
\begin{equation}\label{def:regularizedFunctional}
	\boI_\varepsilon(q)=\inf_{p\in\Hilb}\left\{\frac1\varepsilon\|p-q\|^2_{\Hilb} + \boI(p)\right\},\ \hbox{ for all }q\in\Hilb.
\end{equation}
We recall here the celebrated regularizing lemma for non-convex functionals due to Ekeland and Lasry \cite[Lemma 7]{Eke-Las}, which applies directly to any $\boI\in\X(\boK)$ satisfying \eqref{hyp:convex}. Adapted to this framework, it can be partially stated as follows:
\begin{proposition}\label{pro:regularizing} Let $\boI\in\X(\boK)$ satisfy \eqref{hyp:convex}. Then there exists $\varepsilon_0>0$ such that \eqref{def:regularizedFunctional} is finite-valued and of class $\boC^1(\Hilb;\R)$,  for all $\varepsilon\in(0,\varepsilon_0)$. Moreover, $\boI'_\varepsilon:\Hilb\to\Hilb^*$ is globally Lipschitz, and the following properties hold:
	\begin{enumerate}
		\item[(i)]$ \boI(q)\geq \boI_\varepsilon(q) \geq \inf_{p\in\Hilb} \boI(p)$, for all $q\in\Hilb$. 
		\item[(ii)] The critical points of $\boI$ and $\boI_\varepsilon$ are characterized by
		\begin{equation}\label{def:CriticalPoint-Iepsilon}
			\Sigma(\boI)=\Sigma(\boI_\varepsilon)=\{q\in\Hilb:\ \boI(q)=\boI_\varepsilon(q)\}.
		\end{equation}
	\end{enumerate}
\end{proposition}

A direct consequence of this result is that any global minimizer $q$ of $\boI$  satisfies
\begin{equation*}
	\min_{\Hilb}\boI=\boI(q)=\boI_\varepsilon(q)=\min_{\Hilb}\boI_\varepsilon,\quad\hbox{for all $\varepsilon\in(0,\varepsilon_0)$}.
\end{equation*}

\begin{remark}\label{remark:Lipschitz}
By the Riesz representation theorem we may identify $\Hilb^*$ with $\Hilb$, under which Proposition~\ref{pro:regularizing} implies that the map $\boI'_\varepsilon:\Hilb\to\Hilb$ is globally Lipschitz continuous.
\end{remark}

The following lemma shows that, on compact sets, $\mathcal{I}_\varepsilon$ admits $\varepsilon$-uniform lower bounds in terms of the minimum of $\mathcal{I}$ in the set. In particular, this improves assertion (i) of Proposition~\ref{pro:regularizing} when restricted to compact sets, which will be crucial in the subsequent analysis.

\begin{lemma}\label{lemma:bound-regularized}
Let $\boI\in\X(\boK)$ be bounded from below and satisfy \eqref{hyp:convex}. Let $\boA\subset\boK$ be a compact set in $\Hilb$ such that
\begin{equation}
\min_\boA\boI>\inf_\Hilb\boI.
\end{equation}
Then, for every $\delta\in(0,\min_{\boA}\boI-\inf_{\Hilb}\boI)$, there exists $\varepsilon^*:=\varepsilon^*(\delta)\in(0,\varepsilon_0)$ such that
\begin{equation}\label{eq:bounds_boA}
\boI(q)\geq\boI_\varepsilon(q)\geq\min_\boA\boI-\delta,\ \hbox{ for all }q\in\boA,\hbox{ and all }\varepsilon\in(0,\varepsilon^*),
\end{equation}
where $\varepsilon_0>0$ is given by Proposition~\ref{pro:regularizing}.
\end{lemma}
\begin{proof}
Let us first show that $\inf_{\boA}\boI$ is achieved. To this end, take $\{q_n\}\subset \boA$ with $\boI(q_n)\to \inf_{\boA}\boI$ as $n\to\infty$. Since $\boA$ is compact in $\Hilb$, there exists $q\in\boA$ such that $q_n\to q$ as $n\to\infty$, up to a subsequence. Using that $\boI$ is lower semicontinuous with respect to $\Hilb$, we derive
\begin{equation}
\boI(q)\leq\liminf_{n\to\infty}\boI(q_n)=\inf_{\boA}\boI.
\end{equation}
Thus, $\boI(q)=\min_{\boA}\boI$ necessarily. 

Now we take $\delta\in(0,\min_\boA\boI-\inf_\Hilb\boI)$ and we claim that, for some $\varsigma:=\varsigma(\delta)>0$, the following holds: 
\begin{equation}\label{Claim-proof2}
	\boI(p)>\min_\boA\boI-\delta,\quad\hbox{for all }p\in\boK\hbox{ with }\max_{q\in\boA}\|p-q\|^2_\Hilb<\varsigma.
\end{equation}
Otherwise, there exist sequences $\{q_n\}\subset\boA$ and $\{p_n\}\subset\boK$ such that
\begin{equation}
 \|p_n-q_n\|_\Hilb^2<\dfrac{1}{n},\ \hbox{ and }\ \boI(p_n)\leq \min_\boA\boI-\delta,\quad\hbox{for all }n\in\mathbb{N}.
\end{equation}
Since $\boA$ is compact, up to a subsequence,  $\|q_n-q_0\|_\Hilb\to 0$ as $n\to\infty$, for some $q_0\in \boA$, and hence $\lim_{n\to\infty}\|p_n-q_0\|_\Hilb=0.$ Then, the lower semicontinuity of $\boI$ with respect to $\|\cdot\|_\Hilb$ implies that
\begin{equation}
	\boI(q_0)\leq\liminf_{n\to\infty}\boI(p_n)\leq\min_\boA\boI-\delta<\min_\boA\boI,
\end{equation}
which is a contradiction, hence \eqref{Claim-proof2} holds.

To conclude, fix both $\varepsilon\in(0,\varepsilon_0)$ and $q\in\boA$, and observe that \eqref{def:regularizedFunctional} recasts as follows:
\begin{equation}
	\boI_\varepsilon(q)=\inf_{p\in\Hilb}\Phi_\varepsilon^q(p),\quad\hbox{where }\quad\Phi_\varepsilon^q(p)=\dfrac{1}{\varepsilon}\|p-q\|_\Hilb^2+ \boI(p),\quad\hbox{for all }(p,q)\in\Hilb\times\Hilb.
\end{equation}
Since
\begin{equation}
	\Phi_\varepsilon^q(p)\geq\dfrac{\varsigma}{\varepsilon}+\inf_\Hilb\boI,\quad\hbox{when }\|p-q\|_\Hilb^2\geq\varsigma,
\end{equation}
it is immediate that
\begin{equation}\label{proof:ineqPhi2}
	\Phi_\varepsilon^q(p)>\min_\boA\boI-\delta\quad\hbox{when }\|p-q\|_\Hilb^2\geq\varsigma,\quad\hbox{and}\quad\varepsilon<\varepsilon^*:=\dfrac{\varsigma}{\min_\boA\boI-\inf_\Hilb\boI-\delta}.
\end{equation}
Moreover, \eqref{Claim-proof2} yields
\begin{equation}\label{proof:ineqPhi1}
	\Phi_\varepsilon^q(p)>\min_\boA\boI-\delta,\quad\hbox{when }\|p-q\|_\Hilb^2<\varsigma.
\end{equation}
Then, the proof is completed by taking the infimum over $p\in\Hilb$ in the last two inequalities.
\end{proof}
\begin{remark}
Recall that \eqref{def:CriticalPoint-Iepsilon} holds for every $\varepsilon\in(0,\varepsilon_0)$. Hence, if $\boA$ contains a critical point $q$ with level $\min_{\boA}\boI$, then $\boI_\varepsilon(q)=\min_{\boA}\boI$ for all $\varepsilon\in(0,\varepsilon_0)$. In that case, one has  $\inf_{\delta>0}\varepsilon^*(\delta)>0$.
\end{remark}
The previous lemma is a fairly generic property of the  regularized functionals. In particular, it provides uniform estimates on compact sets in $\Hilb$ (for all sufficiently small $\varepsilon$) for $\mathcal{I}_\varepsilon$ whenever the functional admits global minimizers. This will be essential in the following sections. Complementarily, by condition (i) in Proposition~\ref{pro:regularizing}, the negative sign hypothesis of $\boI$ in \eqref{hyp:Splitting} is transferred to $\mathcal{I}_\varepsilon$, for all $\varepsilon>0$. Both facts are gathered in the following result.
\begin{corollary}\label{cor-SplittingRegularized}
Let $\boI\in\X(\boK)$ be bounded from below and satisfy \eqref{hyp:convex}. Assume that
\begin{equation}\label{H:negative_levels}
0\geq\boI(q)>\inf_\Hilb\boI,\ \hbox{for all }q\in\Hilb_1\ \hbox{ with }\|q\|_\Hilb\in[0,r],
\end{equation}
for some $r>0$. Fix a closed interval $I\subset(0,r]$, and define $
   A(I):=\{q\in\Hilb_1:\ \|q\|_\Hilb\in I\}$. Then, for every $\sigma\in(0,\min_{A(I)}\boI-\inf_\Hilb\boI)$, there exists $\varepsilon^*:=\varepsilon^*(\sigma)\in(0,\varepsilon_0)$ such that
\begin{equation}\label{ineq:lemma}
	0\geq\boI(q)\geq\boI_\varepsilon(q)\geq \inf_\Hilb\boI+\sigma,\quad\hbox{for all }q\in A(I),\hbox{ and all }\varepsilon\in(0,\varepsilon^*),
\end{equation}
where $\varepsilon_0>0$ is given by Proposition~\ref{pro:regularizing}. 
\end{corollary}

\begin{remark}
Notice that $A(I)$ is a closed and bounded set in the finite-dimensional space $\Hilb_1$, so it is compact in $\Hilb$. Therefore, Corollary~\ref{cor-SplittingRegularized} follows directly from Lemma~\ref{lemma:bound-regularized} by identifying $\boA=A(I)$ and $\delta=-\sigma+\min_\boA\boI-\inf_\Hilb\boI$. The formulation with $\sigma$ instead of $\delta$ is adopted here for convenience according  with the subsequent results. 
\end{remark}

To conclude, we introduce the Lipschitz map $\omega:\Hilb\to\boK$  from \cite[pg. 310]{Eke-Las}, which satisfies
\begin{equation}\label{def:omega}
	\boI(\omega(q))=\boI_\varepsilon(q)-\dfrac{1}{\varepsilon}\|q-\omega(q)\|_\Hilb^2,\quad\hbox{and}\quad   \boI_\varepsilon'(q)=\frac{2}{\varepsilon}(q-\omega(q)),\quad \hbox{for all }q\in\Hilb.
\end{equation}
In addition,
\begin{equation}\label{def:subdifferential1}
	\boI'_\varepsilon(q)-\boF'(\omega(q))\in\partial\Psi(\omega(q)),\quad\hbox{for all }q\in\Hilb,
\end{equation}
where $\partial\Psi$ denotes the sub-differential of the convex functional $\Psi:\Hilb\to(-\infty,\infty]$. Equivalently, \eqref{def:subdifferential1} may be written as follows:
\begin{equation}\label{def:subdifferential2}
	\Psi(p)-\Psi(\omega(q))-\langle \boI'_\varepsilon(q)-\boF'(\omega(q)),p-\omega(q)\rangle_\Hilb\geq0,\quad\hbox{for all }q,p\in\Hilb.
\end{equation}
The next result shows the relation between the classical Palais--Smale sequences of $\mathcal{I}_\varepsilon$ and those of $\mathcal{I}$ from Definition~\ref{def:PalaisSmale}.
\begin{proposition}\label{Pro:EquivalencePalaisSmale}
	Let $\boI\in\X(\boK)$ satisfy \eqref{hyp:convex}. Let $\{q_n\}\subset\Hilb$ be a classical Palais--Smale sequence at level $c\in\R$ for the regularized functional $\boI_\varepsilon$, namely
	\begin{equation}\label{def:PS-Classical}
		\lim_{n\to\infty}\boI_\varepsilon(q_n)=c,\quad\lim_{n\to\infty}\|\boI'_\varepsilon(q_n)\|_\Hilb=0,
	\end{equation}
    where $\varepsilon\in (0,\varepsilon_0)$ and $\varepsilon_0>0$ is given by Proposition~\ref{pro:regularizing}. Then there exists a Palais--Smale sequence $\{p_n\}\subset\boK$ for $\boI$ at the same level with $\lim_{n\to\infty}\|q_n-p_n\|_\Hilb\to 0$.	
\end{proposition}
\begin{proof}
Let $\{q_n\}\subset\Hilb$ satisfy \eqref{def:PS-Classical} for some $c\in\R$, and define $\epsilon_n:= \|\boI_\varepsilon'(q_n)\|_\Hilb$. Since $\lim_{n\to\infty}\epsilon_n=0$, by \eqref{def:omega} we obtain $\lim_{n\to\infty}\boI(\omega(q_n))=c.$ Moreover, \eqref{def:subdifferential2} yields
\begin{equation*}\label{Proof-Ps2}
	\Psi(p)-\Psi(\omega(q_n))+\langle \boF'(\omega(q_n)),p-\omega(q_n)\rangle_\Hilb\geq-\epsilon_n\|p-\omega(q_n)\|_\Hilb,\quad\hbox{for all }p\in\Hilb.
\end{equation*}
We conclude the proof by taking $p_n=\omega(q_n)$.
\end{proof}

\begin{remark}\label{remark:difficulties}
By Proposition~\ref{pro:regularizing}, $\boI$ and $\boI_\varepsilon$ have the same critical points. Therefore, a natural strategy to prove Theorem~\ref{MainTheorem} is to adapt the argument in \cite[Theorem 4]{Bre-Nir} to the regularized functional $\boI_\varepsilon$, while imposing the corresponding assumptions on $\boI$. In that case, a strong (in $\Hilb$) Palais--Smale condition for $\boI_\varepsilon$ is required to obtain the min-max geometry detailed in Section~\ref{Sect:Scheme}. Moreover, by Proposition~\ref{Pro:EquivalencePalaisSmale} it follows that the strong convergence in $\Hilb$ for the Palais--Smale sequences of $\boI$ implies such compactness condition for the classical Palais--Smale sequences of $\boI_\varepsilon$. However, this does not hold generally, since only $\tau$-convergence is guaranteed by $\PS$. This subtle issue renders this approach ineffective for proving Theorem~\ref{MainTheorem}. The key idea is therefore to work directly with $\boI$, while still exploiting the regularity properties of $\boI_\varepsilon$ within the argument, see Section~\ref{Sect:Negative-Flow} below.
\end{remark}

\subsection{The negative gradient flow}\label{Sect:Negative-Flow}

In this section we address the technical difficulties described in Remark~\ref{remark:difficulties} and construct a continuous map with the properties required to establish the desired mountain-pass geometry (see Section~\ref{sect:min-max} for the min-max argument). The construction relies crucially on the results of Sections~\ref{Sect:tau-levelsets} and \ref{Sect:eke-las}.

Throughout the section we consider a functional $\boI\in\X(\boK)$ bounded from below with $\inf_\Hilb\boI<0$, satisfying $\PS$ and \eqref{hyp:convex}. Under these hypotheses, the existence of a global minimizer $\minq\in\boK$ is guaranteed by Lemma~\ref{lemma:abstract-minimizer}. We also assume that $\minq\in\boK$ is the only critical point of $\boI$ at a negative level, i.e.
\begin{equation}\label{H:existence_unique_minimizer}
    \Sigma(\boI)\cap \{q\in\Hilb: \boI(q)<0\}=\{\minq\},\quad\boI(\minq)=\min_\Hilb\boI,
\end{equation}
where the set $\Sigma(\boI)$ is introduced in \eqref{eq:critical_point_set}. Moreover, there exists $r\in(0,\|\minq\|_\Hilb)$ satisfying \eqref{H:negative_levels} and
\begin{equation}\label{H:zero_isolated}
  \Sigma(\boI)\cap\{q\in\Hilb_1: \|q\|_\Hilb\in(0, r]\}=\emptyset.
\end{equation}

\begin{remark}
By Lemma~\ref{lemma:abstract-minimizer}, the proof for Theorem~\ref{MainTheorem} is reduced to establish the existence of a second nontrivial critical point for $\boI$. Therefore, both assumptions \eqref{H:existence_unique_minimizer} and \eqref{H:zero_isolated} entail no loss of generality. Observe also that \eqref{H:negative_levels} holds trivially under hypotheses  \eqref{hyp:Splitting} and \eqref{H:existence_unique_minimizer}, hence assuming \eqref{H:negative_levels} imposes no additional restriction in the proof of Theorem~\ref{MainTheorem}. We do not assume the full splitting condition \eqref{hyp:Splitting} at this stage, since the positivity condition on $\Hilb_2$ will only be needed in the next subsection.
\end{remark}

Let us denote the sphere in $\Hilb_1$ of radius $\rho>0$ by
\begin{equation}\label{def:Sphere}
    S_\rho:=\{q\in\Hilb_1:\ \|q\|_\Hilb=\rho\}.
\end{equation}
As Corollary~\ref{cor-SplittingRegularized} applies to $A(\{r\})=S_r$, we may define $\sigma_r:=\min_{S_r}\boI-\boI(\minq)>0$ and fix the constants
\begin{equation}\label{def:Constant}
\sigma\in(0,\sigma_r),\quad\quad\varepsilon\in(0,\varepsilon^*(\sigma)).
\end{equation}
Therefore, condition \eqref{H:zero_isolated}, together with Proposition~\ref{pro:regularizing} and Corollary~\ref{cor-SplittingRegularized}, yields
\begin{equation}\label{cond:non-degeneracy}
	0\geq\boI(y)>\boI_\varepsilon(y)\geq\boI(\minq)+\sigma;\quad\boI'_\varepsilon(y)\neq0,\quad\hbox{for all }y\in S_r,
\end{equation}
where the strict inequality follows from the fact that $S_r$ contains no critical points.

Let us now define the functional ${\boG}:\Omega\to\Hilb$ by
\begin{equation}\label{def:F-PVI}
	{\boG}(q)=-\frac{\boI'_\varepsilon(q)}{\|\boI_\varepsilon'(q)\|_\Hilb^2},\quad\text{for all }q\in \Omega:=\{q\in \Hilb:\, \|\boI'_\varepsilon(q)\|_\Hilb\not=0\}.
\end{equation}
In particular, ${\boG}$ is well defined on $S_r$ by condition \eqref{H:zero_isolated}. We also emphasize that, in \eqref{def:F-PVI}, each $\boI'_\varepsilon(q)$ is identified with an element of $\Hilb$ through the Riesz representation theorem.

The following lemma establishes the existence and regularity properties of the negative gradient flow $x(t;y)$ associated with $\boI'_\varepsilon$, for sufficiently small nonzero initial data $y\in\Hilb_1$. This result is partially analogous to \cite[Lemma 4]{Bre-Nir}. 

\begin{lemma}\label{lemma:PVI}
Let $\boI\in\X(\boK)$  satisfy $\PS$, \eqref{hyp:convex}, \eqref{H:negative_levels}, \eqref{H:existence_unique_minimizer}, and  \eqref{H:zero_isolated}, for some $\minq\in\boK$ and $r\in(0,\|\minq\|_\Hilb)$. Let $\sigma$ and $\varepsilon$ be given by \eqref{def:Constant}. Consider the initial value problem
\begin{equation}\label{eq:ivp}
	\frac{dx}{dt}=\boG(x(t)),\quad x(0)=y,\quad\text{with }y\in S_r,
\end{equation}
with ${\boG}:\Omega\to\Hilb$ given in \eqref{def:F-PVI}. Then \eqref{eq:ivp} admits a unique  solution $x(\cdot;y) \in \mathcal{C}^1([0,\taumax(y)); \Omega)$, where
\begin{equation}\label{eq:taumax-bounds}
\taumax(y)=\boI_\varepsilon(y)-\boI(\minq)\in\left[\sigma,-\boI(\minq)\right],\qquad\hbox{for all }y\in S_r.
\end{equation}
Moreover,
\begin{equation}\label{ineq:regularized-flow}
\boI_\varepsilon(x(t;y))=\boI_\varepsilon(y)-t,\quad \hbox{for all } (t,y)\in[0,\taumax(y))\times S_r,
\end{equation}
and
\begin{equation}
\lim_{t\to\taumax(y)}\boI_\varepsilon(x(t;y))=\boI(\minq),\quad\hbox{for all }y\in S_r.
\end{equation}
\end{lemma}

\begin{proof}
By Remark~\ref{remark:Lipschitz}, the operator $\boI'_\varepsilon:\Hilb\to\Hilb$ is globally Lipschitz, thus ${\boG}$ is locally Lipschitz in $\Omega$. Therefore, by standard ODE theory in Banach spaces, for every $y\in S_r$ there exists a unique solution $x=x(\cdot;y)\in\boC^1([0,\taumax(y));\Omega)$ to \eqref{eq:ivp} defined on a maximal interval of existence $[0,\taumax(y))$ (see \cite{Deim} for details). Moreover, the chain rule leads to
\begin{equation}\label{eq:derivative-1}
	\frac{d}{dt}(\boI_\varepsilon(x(t;y))=\left\langle \boI_\varepsilon'(x(t;y)),\frac{dx}{dt}(t;y)\right\rangle_\Hilb=-1,\quad\hbox{for all } (t,y)\in[0,\taumax(y))\times S_r,
\end{equation}
which yields relation \eqref{ineq:regularized-flow}. In particular, from \eqref{cond:non-degeneracy} it follows that the function $\boI_\varepsilon(x(\cdot;y)):[0,\taumax(y))\to\R$ is strictly decreasing and negative, for every $y\in S_r$.

Now we claim that  $\boI_\varepsilon(x(t;y))\to\boI(\minq)$, as $t\to\taumax(y)$, for all $y\in S_r$. In particular, passing to the limit in \eqref{ineq:regularized-flow}, this yields the characterization of $\taumax(y)$ in \eqref{eq:taumax-bounds} which, together with the estimates in \eqref{ineq:lemma}, directly implies the uniform bounds on $y$ for $\taumax$ given in \eqref{eq:taumax-bounds}. Therefore, the proof is completed if the claim is true. 

To prove it, fix $y\in S_r$ and let ${t_n}\subset(0,\taumax(y))$ be an increasing sequence converging to $\taumax(y)$. Since ${\boI_\varepsilon(x(t_n;y))}$ is decreasing and bounded from below, we have that
\[
\boI_\varepsilon(x(t_n;y))\to c\quad\text{as }n\to\infty,\quad\text{for some }c\in[\boI(\minq),0).
\]
Moreover, 
\[
\|\boI_\varepsilon'(x(t_n;y))\|_\Hilb\to 0\quad\text{as }n\to\infty,
\]
necessarily, otherwise the solution $x(\cdot;y)$ could be extended beyond $\taumax(y)$, contradicting its maximality. Therefore $\{x(t_n;y)\}$ is a classical Palais--Smale sequence for $\boI_\varepsilon$ and, by Proposition~\ref{Pro:EquivalencePalaisSmale}, there exists a Palais--Smale sequence $\{p_n\}\subset\boK$ for $\boI$ at level $c$. Since $c\neq0$ and $\PS$ holds, this yields a critical point $q\in\boK$ at level $c$. Necessarily, $q=\minq$ and $c=\boI(\minq)$, since $\minq$ is the only critical point at a negative level. Finally, by (ii) in Proposition~\ref{pro:regularizing}, we conclude that $\boI_\varepsilon(x(t;y))\to \boI_\varepsilon(\minq)=\boI(\minq)$ as $t\to\taumax(y)$. Thus the claim holds and the proof is completed.
\end{proof}

For any $\delta\in(0,\sigma)$, define the continuous function $\lambda_\delta:S_r\to\R$ by
\begin{equation}\label{def:lambda}
\lambda_\delta(y)=\boI_\varepsilon(y)-\boI(\minq)-\delta,
\end{equation}
which, together with \eqref{ineq:regularized-flow}, yields
\begin{equation}\label{eq:charact-lambda-flow}
\boI_\varepsilon(x(\lambda_\delta(y);y))=\boI_\varepsilon(y)-\lambda_\delta(y)=\boI(\minq)+\delta,\quad\hbox{for all }y\in S_r.
\end{equation}
Furthermore, from \eqref{ineq:lemma} we obtain the following estimates:
\begin{equation}\label{eq:lambda-bounds}
	0<\sigma-\delta\leq\lambda_\delta(y)\leq -\boI(\minq)-\delta,\qquad\hbox{for all }y\in S_r.
\end{equation}

\begin{corollary}\label{cor:continuity}
Let $\boI\in\X(\boK)$ satisfy $\PS$, \eqref{hyp:convex}, \eqref{H:negative_levels}, \eqref{H:existence_unique_minimizer}, and \eqref{H:zero_isolated}, for some $\minq\in\boK$ and $r\in(0,\|\minq\|_\Hilb)$. Let $\sigma$ and $\varepsilon$ be given by \eqref{def:Constant}. Fix $\delta\in(0,\sigma)$ and define the sets 
\begin{equation}\label{def:LambdaSet}
    \Lambda^r_\delta:=\{(t,y):\  t\in[0,\lambda_\delta(y)]\hbox{ and } y\in S_r\},\quad\Omega^r_\delta:=\{z(t,y):\ (t,y)\in\Lambda^r_\delta\}.
\end{equation} Then, the map $z:\Lambda^r_\delta\to\Omega^r_\delta$ defined by $z(t,y)=x(t;y)$ is Lipschitz. 
\end{corollary}
\begin{proof}
We claim that
\begin{equation}
m:=\inf_{\Omega^r_\delta}\|\boI'_\varepsilon(\cdot)\|_\Hilb=\inf_{\Lambda_\delta^r}\|\boI'_\varepsilon(z(\cdot))\|_\Hilb>0.
\end{equation} 
Indeed, if $m=0$, there is a sequence $\{(t_n,y_n)\}\subset\Lambda_\delta^r$ such that $\lim_{n\to\infty}\|\boI_\varepsilon'(z(\cdot))\|_\Hilb=0$. Since $\Lambda_\delta^r$ is compact, there exists $(t_0,y_0)\in\Lambda_\delta^r$ such that, up to subsequence, 
\begin{equation}
    |t_n-t_0|+\|y_n-y_0\|_\Hilb\to0,\quad\hbox{ as }n\to\infty.
\end{equation}
Then, since $\boI_\varepsilon$ is continuous, the characterization \eqref{ineq:regularized-flow} of the flow $\boI_\varepsilon(z(\cdot))$ yields
\begin{equation}
    \lim_{n\to\infty}\boI_\varepsilon(z(t_n,y_n))=\lim_{n\to\infty}(\boI_\varepsilon(y_n)-t_n)=\boI_\varepsilon(y_0)-t_0=\boI_\varepsilon(z(t_0,y_0)).
\end{equation}
Moreover, by \eqref{cond:non-degeneracy} and \eqref{eq:charact-lambda-flow} we obtain
\begin{equation}
    0>\boI_\varepsilon(y_0)\geq\boI_\varepsilon(z(t_0,y_0))\geq\boI_\varepsilon(z(\lambda_\delta(y_0),y_0))=\boI(\minq)+\delta.
\end{equation}
In particular, we have obtained a Palais--Smale sequence for $\boI_\varepsilon$ at a level $c\in(\boI(\minq),0)$. By Proposition~\ref{Pro:EquivalencePalaisSmale} and $\PS$,  there exists a critical point of $\boI$ at level $c$, contradicting \eqref{H:existence_unique_minimizer}. Hence the claim $m>0$ holds. As a consequence,
\begin{equation}\label{eq:F-estimate}
\sup_{q\in\Omega^r_\delta}\|{\boG}(q)\|_\Hilb\leq m^{-1}.
\end{equation}
It is straightforward to verify that there exists a constant $L>0$ such that
\begin{equation}\label{eq:Gbounded}
\|\boG(q)-\boG(p)\|_\Hilb\leq L\|q-p\|_\Hilb,\quad\text{for all }q,p\in\Omega_\delta^r.
\end{equation}
Therefore, Gronwall's lemma yields
\begin{equation}\label{eq:gronwall}
    \|z(t,y)-z(t,\hat y)\|_\Hilb\leq e^{L\min\{\lambda_\delta(y),\lambda_\delta(\hat y)\}}\|y-\hat y\|_\Hilb \leq e^{-L(\boI(\minq)+\delta)}\|y-\hat y\|_\Hilb,
\end{equation}
for every $y,\hat y\in S_r$ and $t\in [0,\min\{\lambda_\delta(y),\lambda_\delta(\hat y)\}]$. Thus, let us fix $(t,y),(\hat t,\hat y)\in \Lambda_\delta^r$ and assume without loss of generality that $\hat t\leq t$, so that $\hat t\in [0,\min\{\lambda_\delta(y),\lambda_\delta(\hat y)\}]$. Applying \eqref{eq:Gbounded} and \eqref{eq:gronwall}, we derive
\begin{align*}
    &\|z(t,y)-z(\hat t,\hat y)\|_\Hilb \leq \|z(t,y)-z(\hat t,y)\|_\Hilb + \|z(\hat t,\hat y)-z(\hat t,y)\|_\Hilb
    \\
    &\leq\int_{\hat t}^t\left\|\boG(z(s,y))\right\|ds + e^{-L(\boI(\minq)+\delta)}\|y-\hat y\|_\Hilb \leq m^{-1}|t- \hat t|+e^{-L(\boI(\minq)+\delta)}\|y-\hat y\|_\Hilb.
\end{align*}
This proves that $z$ is Lipschitz.
\end{proof}

    By Corollary~\ref{cor:continuity}, the flow $x(t;y)$ reaches the negative level set $\{q\in\boK:\boI(q)=\boI(\minq)+\delta\}$, with continuous dependence on the initial datum $y$, for every $\delta\in(0,\sigma)$. Nevertheless, this does not resolve the issues described both in Remark~\ref{remark:difficulties} and in Section~\ref{Sect:Scheme}, which constitutes the central difficulty of the paper. We overcome this by exploiting the properties of the Lipschitz map $\omega:\Hilb\to\boK$ defined in \eqref{def:omega}, which may be viewed as a smooth $\varepsilon$-perturbation of the identity. The precise construction is given in the following result.

\begin{proposition}\label{Pro:Gamma}
Let $\boI\in\X(\boK)$ satisfy $\PS$, \eqref{hyp:convex}, \eqref{H:negative_levels}, \eqref{H:existence_unique_minimizer}, \eqref{H:zero_isolated}, for some $\minq\in\boK$ and $r\in(0,\|\minq\|_\Hilb)$. Let $\sigma$ and $\varepsilon$ be given by \eqref{def:Constant}. Then there exists $\delta\in(0,\sigma)$ such that the map $\gamma^*:[0,1]\times S_r\to\boK$, defined by
\begin{equation}\label{def:gamma*}
    \gamma^*(s,y):=\left\{\begin{array}{cl}
         \omega(x(2s\lambda_\delta(y);y)),&\hbox{for }s\in\left[0,\tfrac{1}{2}\right],  \vspace{1mm} \\ 
         (2s-1)\minq+2(1-s)\omega(x(\lambda_\delta(y);y)),&\hbox{for } s\in\left[\tfrac{1}{2},1\right],
    \end{array}\right.
\end{equation}    
is continuous, where $\omega:\Hilb\to\boK$ and $\lambda_\delta:S_r\to\R$ are given by \eqref{def:omega} and \eqref{def:lambda}, respectively. Moreover,
\begin{equation}
\label{eq:gamma-negativeflow}
    \boI(\gamma^*(s,y))\leq0,\quad  \hbox{for all }(s,y)\in[0,1]\times S_r,
\end{equation}and there exists $\rho_\delta\in (0,r)$ satisfying the following:
\begin{equation}\label{eq:gamma-notzero}
    \|\gamma^*(s,y)\|_\Hilb\geq \rho_\delta,\quad \hbox{for all }(s,y)\in[0,1]\times S_r.
\end{equation}
\end{proposition}
\begin{proof}
Let $\delta\in (0,\sigma)$ be chosen small enough later. First, recall from Remark~\ref{remark:Lipschitz} and \eqref{def:omega} that $\omega:\Hilb\to\boK$ is a Lipschitz operator. Thus, the fact that $\gamma^*(s,y)\in\boK$ for every $s\in [0,1/2]$ is obvious. In addition, since $\minq\in\boK$ and $\boK$ is convex, it follows that $\gamma^*(s,y)\in\boK$ for every $s\in (1/2,1]$ as well. Moreover, note that we may write 
\begin{equation}
    \gamma^*(s,y)=\omega(z(2s\lambda_\delta(y),y)),\quad\hbox{for all }(s,y)\in\left[0,\tfrac{1}{2}\right]\times S_r,
\end{equation}where $z:\Lambda_\delta^r\to\Omega_\delta^r$ is introduced in Corollary~\ref{cor:continuity}. This result  implies that the map $\gamma^*$ defined is continuous. For convenience, we maintain the notation $z(t,y)$ instead of $x(t;y)$ in the rest of the proof.

Let us fix $\alpha=-\boI(\minq)-\sigma$ (note that $\alpha>-\min_{S_r}\boI\geq0$), and let $\eta_\sigma>0$ be given by Lemma \ref{lemma-alpha}. We fix also $\eta\in(0,\min\{\eta_\sigma,\|\minq\|_\Ebanach\})$.

Now we prove \eqref{eq:gamma-negativeflow}. From the definition of $\omega$ in \eqref{def:omega}, and by \eqref{ineq:regularized-flow}, we write
\begin{align*}
    \boI(\omega(z(t,y)))&=\boI_\varepsilon(z(t,y))-\dfrac{1}{\varepsilon}\|z(t,y)-\omega(z(t,y))\|_\Hilb^2
    \\
    &=\boI_\varepsilon(y)-t-\dfrac{1}{\varepsilon}\|z(t,y)-\omega(z(t,y))\|_\Hilb^2,\quad\hbox{for all }(t,y)\in\Lambda_\delta^r.
\end{align*}
Then, using that $\boI_\varepsilon|_{S_r}<0=\boI(0)$, we obtain
\begin{equation}\label{ineq:omega-flow}
\boI(\omega(z(t,y)))<0=\boI(0),\quad\hbox{for all }(t,y)\in\Lambda_\delta^r.
\end{equation}
In particular,
\begin{equation} 
    \boI(\gamma^*(s,y))<0, \quad\hbox{for all }(s,y)\in\left[0,\tfrac{1}{2}\right]\times S_r.
\end{equation}

Let us consider the case $(s,y)\in[1/2,1]\times S_r$. By \eqref{def:omega} and \eqref{eq:charact-lambda-flow}, we deduce
\begin{equation}\label{ineq:proofdelta}
    \boI(\omega(z(\lambda_\delta(y),y)))\leq\boI_\varepsilon(z(\lambda_\delta(y),y))=\boI(\minq)+\delta.
\end{equation}
In other words, $\omega(z(\lambda_\delta(y);y)))\in \{q\in\boK: \boI(q)\leq \boI(\minq)+\delta\}$. Then, for $\eta\in (0,\eta_\sigma)$, Lemma \ref{lemma-LevelSets} provides a number $\delta_\eta>0$ such that, fixing $\delta\in (0,\min\{\sigma,\delta_\eta\})$, one has
\begin{equation}\label{ineq:proofeta}
    \|\omega(z(\lambda_\delta(y),y))-\minq\|_\Ebanach<\eta.
\end{equation}Therefore, by taking $t=2(1-s)$ in the notation of Lemma~\ref{lemma-alpha}, we obtain
\begin{equation}\label{ineq:gamma*-negative}
    \boI(\gamma^*(s,y))\leq\boI(\omega(z(\lambda_\delta(y),y)))+\alpha.
\end{equation}
By \eqref{ineq:proofdelta} and the definition of $\alpha$, it follows that
\begin{equation}\label{ineq:gamma*-negative2}
    \boI(\gamma^*(s,y))\leq\boI_\varepsilon(z(\lambda_\delta(y),y))+\alpha=\delta-\sigma<0.
\end{equation}

To conclude, we focus on proving \eqref{eq:gamma-notzero}. By \eqref{ineq:omega-flow}, $
	\omega(z(t,y))\neq0,$ for all $(t,y)\in\Lambda_\delta^r.$ Hence, the continuity of $\omega(z(\cdot))$ in $\Lambda_\delta^r$ implies
\begin{equation}\label{eq:gamma*-positive1}
    m_\delta:=\min_{\left[0,\tfrac{1}{2}\right]\times S_r
}\|\gamma^*(\cdot)\|_\Hilb= \min_{\Lambda_\delta^r}\|\omega(z(\cdot))\|_\Hilb>0.
\end{equation}
Complementarily, when $(s,y)\in\left[\tfrac{1}{2},1\right]\times S_r$, the embedding \eqref{eq:embedding} yields
\begin{equation}
    c_0\|\gamma^*(s,y)\|_\Hilb\geq\|\gamma^*(s,y)\|_{\textnormal E}\geq\|\minq\|_{\textnormal E}-2(1-s)\|\omega(z(\lambda_\delta(y),y))-\minq\|_{\textnormal E},
\end{equation}
where $c_0$ comes from \eqref{eq:embedding}. Then, by \eqref{ineq:proofeta} we conclude that
\begin{equation}\label{eq:gamma*-positive2}
    c_0\|\gamma^*(s,y)\|_\Hilb\geq\|\minq\|_{\textnormal E}-\eta>0,
\end{equation}
and \eqref{eq:gamma-notzero} is satisfied with $\rho_\delta:=c_0^{-1}\min\{\|\minq\|_{\textnormal E}-\eta,m_\delta\}>0$. In particular, note that we may consider $\rho_\delta<r$, which completes the proof.
\end{proof}

\begin{figure}[ht]
\centering
\begin{tikzpicture}[scale=0.7, every node/.style={transform shape},
    node distance=1.6cm,
    param/.style={
        circle, draw, thick,
        minimum size=0.9cm,
        font=\large
    },
    arrow/.style={
        ->, thick, >=stealth
    }
]

\node[param] (r)       {$r$};
\node[param] (sigma)   [right=of r]                {$\sigma$};
\node[param] (eps)     [above right=0.8cm and 1.6cm of sigma]  {$\varepsilon$};
\node[param] (eta)     [below right=0.8cm and 1.6cm of sigma]  {$\eta$};
\node[param] (delta)   [right=of eta]              {$\delta$};
\node[param] (rho)     [right=of delta]            {$\rho$};

\draw[arrow] (r)     -- (sigma);
\draw[arrow] (sigma) -- (eps);
\draw[arrow] (sigma) -- (eta);
\draw[arrow] (eta)   -- (delta);
\draw[arrow] (delta) -- (rho);

\draw[dashed, rounded corners, gray]
    ($(eps.north west)+(-0.3,0.3)$) rectangle ($(rho.south east)+(0.3,-0.3)$);

\end{tikzpicture}
\caption{Chained dependence of parameters.}
\label{fig:param-chain}
\end{figure}
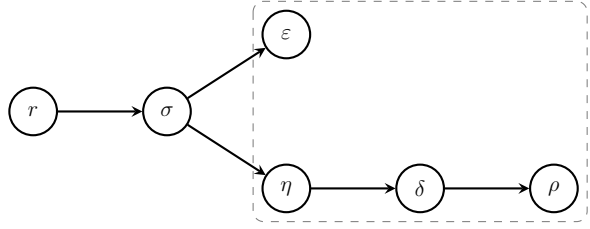
\subsection{Min-max argument}\label{sect:min-max}

We conclude by proving Theorem~\ref{MainTheorem}. The proof relies on a last lemma, whose statement requires some preliminaries. 

Let $\boI\in\X(\boK)$ satisfy the hypotheses of Theorem~\ref{MainTheorem}. By Lemma~\ref{lemma:abstract-minimizer}, there exists a global minimizer $\minq\in\boK$ with $\boI(\minq)=\min_{\Hilb}\boI<0$. We may also consider that \eqref{hyp:Splitting}  holds with $r\in(0,\|\minq\|_\Hilb)$, without loss of generality, since the local linking is a local property. Moreover, since our aim is to prove the existence of a second nonzero critical point for $\boI$, neither do we lose generality by assuming that \eqref{H:existence_unique_minimizer} and \eqref{H:zero_isolated} hold, after taking $r$ sufficiently small. Summarizing, and recalling that \eqref{hyp:Splitting} implies \eqref{H:negative_levels}, it follows that $\boI$ satisfies the hypotheses of Lemma~\ref{lemma:PVI} and Proposition~\ref{Pro:Gamma}. Therefore, let us fix $\sigma$ and $\varepsilon$ as in \eqref{def:Constant}, and then take $\delta$ and $\rho_\delta$ as in Proposition~\ref{Pro:Gamma}.

On the other hand, take a vector $v\in\Hilb_2$ with $|v|_\Hilb=r$, and consider the set $\boN$ defined in \eqref{def:Set-N}, whose boundary $\partial\boN$ splits as in \eqref{def:BoundarySet-N}. Since $\Hilb=\Hilb_1\oplus\Hilb_2$, the decomposition $q=sv+y$ given in \eqref{def:BoundarySet-N} is unique. As a consequence, by Proposition~\ref{Pro:Gamma}, the map $\gamma_0^*:\partial\boN\to\boK$ given by
\begin{equation}\label{def:gamma0}
    \gamma_0^*(q)=q,\text{ if }q\in\boN_1,\quad\gamma_0^*(q)=\gamma_0^*(sv+y):=\gamma^*(s,y),\text{ if }q\in \boN_2,
\end{equation}
is well defined, continuous with respect to $\|\cdot\|_\Hilb$, and satisfies the estimates \eqref{eq:gamma-negativeflow} and \eqref{eq:gamma-notzero} for all $q\in\partial\boN$.

\begin{lemma}\label{Lemma:MountainPass}
Let $\boI\in\X(\boK)$ satisfy $\PS$, \eqref{hyp:Splitting}, \eqref{hyp:convex}, \eqref{H:existence_unique_minimizer}, and \eqref{H:zero_isolated}, for some $\minq\in\boK$ and $r\in(0,\|\minq\|_\Hilb)$. Let $\sigma$ and $\varepsilon$ be given by \eqref{def:Constant}, and $\delta$ by Proposition~\ref{Pro:Gamma}. Then the map $\gamma_0^*:\partial\boN\to\boK$ defined in \eqref{def:gamma0} satisfies
\begin{equation}\label{eq:MountainPassCondition}
\sup_{q\in\partial\boN}\boI(\gamma_0^*(q))=0<\inf_{\gamma\in\Gamma}\sup_{q\in\boN}\boI(\gamma(q))<\infty,
\end{equation}
where $
    \Gamma:=\{\gamma:\boN\to\Hilb\ \hbox{ continuous and such that }\gamma|_{\partial\boN}=\gamma_0^*\}.$
\end{lemma}
\begin{proof}
First, we recall that \cite[Lemma 3]{Bre-Nir} applies directly to the family $\Gamma$. That is, for all $\gamma\in\Gamma$ and all $\rho\in(0,\rho_\delta)$, there exists  $u:=u(\gamma,\rho)\in\boN$ such that
\begin{equation}\label{eq:intersectionBN}
\gamma(u)\in\Hilb_2\quad\hbox{and}\quad\|\gamma(u)\|_\Hilb=\rho.
\end{equation}
Therefore, since $\rho<\rho_\delta<r$, by  \eqref{hyp:Splitting} we have that
\begin{equation}\label{eq:supremumGamma}
    \sup_{q\in\boN}\boI(\gamma(q))>0,\quad\hbox{for all }\gamma\in\Gamma.
\end{equation}
To conclude the proof, it suffices to show that the infimum in \eqref{eq:MountainPassCondition} is finite, which follows from the existence of a continuous map $\gamma \in \Gamma$ whose image is contained in $\boK$. To this end, recalling Remark~\ref{rem:I(0)}, there exists $r_\sigma \in (0,r)$, sufficiently close to $r$, such that
\begin{equation}
    \sigma\in(0,\min_{A(I)}\boI-\boI(\minq)),
\end{equation}
where $I=[r_\sigma,r]$ and $
A(I):=\{y\in\Hilb_1:\ \|y\|_\Hilb\in I\}.$ 
In particular, Corollary~\ref{cor-SplittingRegularized} applies to the annulus $A(I)$ for the fixed parameters $\sigma$ and $\varepsilon$. Consequently, for the parameter $\delta\in(0,\sigma)$ given by Proposition~\ref{Pro:Gamma}, the function $\lambda_\delta$ defined in \eqref{def:lambda} extends trivially as a continuous function to the domain $A(I)$. Moreover, the sets defined in Corollary~\ref{cor:continuity} can be naturally generalized as follows:
\begin{equation}
\Lambda_\delta:=\bigcup_{j\in[r_\sigma,r]}\Lambda_\delta^j,\qquad\Omega_\delta:=\bigcup_{j\in[r_\sigma,r]}\Omega_\delta^j.
\end{equation}
By the same arguments as in the proof of Corollary~\ref{cor:continuity}, the function $z:\Lambda_\delta\to\Omega_\delta$, defined by $z(t,y)=x(t;y)$, is Lipschitz. Therefore, by composition of Lipschitz functions (see \eqref{def:omega}, Remark~\ref{remark:Lipschitz}), the map $h:[0,1/2]\times A(I)\to\boK$ defined by
\begin{equation}\label{def:function-h}
    h(s,y):=\omega(z(2s\lambda_\delta(y),y)),\quad\hbox{for all }(s,y)\in \left[0,\frac12\right]\times A(I),
\end{equation}
inherits the same regularity. Moreover, $h|_{[0,1/2]\times S_r}=\gamma^*$ by construction.

To finish, let $\gamma_{r_\sigma}: \boN \to \Hilb$ be defined by $\gamma_{r_\sigma}(q) := \gamma_{r_\sigma}(sv + y)$ with
    \begin{equation}\label{def:gamma-example}
      \gamma_{r_\sigma}(q)=\left\{\begin{array}{ll}
            \gamma^*(s,y),&\hbox{for } (s,y)\in[0,1]\times S_r, \\
            \dfrac{\|y\|_\Hilb-r_\sigma}{r-r_\sigma}h(s,y)+\dfrac{r-\|y\|_\Hilb}{r-r_\sigma}y, &\hbox{for }s\in\left[0,\tfrac{1}{2}\right),\ \|y\|_\Hilb\in[r_\sigma,r),\\
            (2s-1)\minq+2(1-s)\left(\dfrac{\|y\|_\Hilb-r_\sigma}{r-r_\sigma}h(s,y)+\dfrac{r-\|y\|_\Hilb}{r-r_\sigma}y\right),&\hbox{for }s\in\left[\tfrac{1}{2},1\right],\ \|y\|_\Hilb\in[r_\sigma,r),\\
            y,&\hbox{for }s\in\left[0,\tfrac{1}{2}\right),\ \|y\|_\Hilb\in[0,r_\sigma),\vspace{1mm}\\
            (2s-1)\minq+2(1-s)y,&\hbox{for }s\in\left[\tfrac{1}{2},1\right],\ \|y\|_\Hilb\in[0,r_\sigma).
        \end{array}\right.
    \end{equation}
From this point, it is straightforward to see that $\gamma_{r_\sigma}\in\Gamma$. Moreover, note that every component is a convex combination of elements in $\boK$, hence $\gamma_{r_\sigma}(q)\in\boK$, for all $q\in\boN$, which ends the proof.
\end{proof}

\begin{proof}[Proof of Theorem~\ref{MainTheorem}]

As commented above, the result is proven once we obtain the existence of a second nonzero critical point. Hence we may restrict to consider that $\boI\in\X(\boK)$ also satisfies \eqref{H:existence_unique_minimizer} and \eqref{H:zero_isolated}.

Let $r\in(0,\|\minq\|)$ be given by \eqref{hyp:Splitting}. Fix the constants $\sigma$ and $\varepsilon$ as in \eqref{def:Constant}, and let $\delta$ and $\rho_\delta$ be given by Proposition~\ref{Pro:Gamma}. Let $\gamma^*$ be defined in \eqref{def:gamma*}. In addition, fix $v\in\Hilb_2$ with $\|v\|_\Hilb=r$, and let $\boN$ denote the set \eqref{def:Set-N}. Observe that $\boN$ is compact, thereby it is a complete metric with respect to the norm $\|\cdot\|_\Hilb$. Therefore, as the inequality \eqref{eq:MountainPassCondition} holds by Lemma~\ref{Lemma:MountainPass}, in this setting we may apply \cite[Theorem 1]{ABT} by choosing $K:=\boN$, $K_0:=\partial\boN$ and $\gamma_0:=\gamma_0^*$, each of them given in \eqref{def:Set-N}, \eqref{def:BoundarySet-N} and \eqref{def:gamma0}, respectively. As a result, there exists a Palais–Smale sequence $\{q_n\} \subset \boK$ for $\boI$ at the positive level $c\in(0,\infty)$ given in \eqref{eq:MountainPassCondition}, i.e.
\begin{equation}
c:=\inf_{\gamma\in\Gamma}\sup_{q\in\boN}\boI(\gamma(q)).
\end{equation}
Then, by  $\PS$, up to a subsequence, $\{q_n\}$ converges in $\tau$ to a critical point of $\boI$ at level $c$, which completes the proof.
\end{proof}

\section{The Lorentz force equation in the absence of scalar potential}\label{Sect:LFE}
	In this section we analyze the dynamics of the Lorentz force equation \eqref{eq:LFE-intro}, where both the speed of light in vacuum and the charge-to-mass ratio are normalized to one, without loss of generality. The electromagnetic field is described in terms of the vector potential $A\in\mathcal{C}^1(\R^4;\R^3)$ by
	\begin{equation}\label{Potential-Fields}
		B(t,x)=\nabla\times A (t,x),\qquad E(t,x)=-\partial_t A(t,x).
	\end{equation}
	 We recall that in \eqref{eq:LFE-intro} (as well as in \eqref{Potential-Fields}) we use the standard notation for the curl ($\nabla\times$), cross product ($\times$), and partial derivative ($\partial_t$). We will also use $\nabla$ to denote the gradient with respect to the spatial variables, and $\partial^2_{k,j}:=\partial_k\partial_j$ for the second-order partial derivatives, for $k,j\in\{t,x_1,x_2,x_3\}$.
	
	Let us fix a period $T>0$. We focus on the line of research initiated in \cite{GLM} concerning the qualitative analysis of the dynamics of Equation \eqref{eq:LFE-intro} in the \textsl{purely relativistic regime}, meaning that all magnetostatic configurations in \eqref{Potential-Fields} are excluded (see Remark~\ref{remark:magnetostatic} below). In \cite[Theorem 2.1]{GLM}, the existence of a non-constant periodic solution for Equation \eqref{eq:LFE-intro} is established by minimizing the associated action functional \eqref{eq:functional} (see the variational details below). The result covers a large class of vector potentials in $\mathcal{C}^1_T(\R^4;\R^3)$, where
	\begin{equation}
		\mathcal{C}^k_T(\R^4;\R^3):=\{A\in\mathcal{C}^k(\R^4;\R^3):\ A(\cdot,x) \hbox{ is $T$-periodic }\forall x\in\R^3\},\quad k\in\mathbb{N}.
	\end{equation} 
	
	We provide in Theorem~\ref{MainTheorem:LFE} the existence of a second non-constant periodic solution as an application of Theorem~\ref{MainTheorem}. To do so, we assume the existence of an equilibrium state $x_0 \in \R^3$ for the electromagnetic force, namely a point where the electric field vanishes. In the setting of \eqref{eq:LFE-intro}, this condition reads
	\begin{equation}\label{def:equilibrium}
		\partial_tA(t,x_0)=0,\hbox{ for all }t\in[0,T].
	\end{equation}
	Moreover, we impose $x_0$ to be an \textsl{isolated equilibrium}, that is, there exists $\varsigma>0$ such that
	\begin{equation}\label{def:isolated2}
		\partial_t A(\cdot,x)\not\equiv 0,\quad\text{for all }x\in\R^3\ \hbox{with }|x-x_0|\in(0,\varsigma].
	\end{equation}
	We denote  the set of isolated equilibria by
	\begin{equation}\label{def:EquilibriaSet}
		\Upsilon(A) := \{x_0 \in \R^3 : x_0 \text{ satisfies both \eqref{def:equilibrium} and \eqref{def:isolated2}}\}.
	\end{equation}

	We now state our main result concerning the dynamics of \eqref{eq:LFE-intro}.
	\begin{theorem}\label{MainTheorem:LFE}
		Let $A\in\boC_T^2(\R^4;\R^3)$ satisfying  $\partial_tA\not\equiv0$ and the following properties:
		\begin{align}
        \label{cond:A-decay}
			&\lim_{|x|\to\infty} \big(|\partial_t A(t,x)| + |\nabla A(t,x)|\big)=0,\quad\text{uniformly in }t,
            \\
            \label{hyp:bound-secondderivatives}
            &\nabla A(t,\cdot)\in W^{1,\infty}(\R^3),\quad\text{with }\sup_{t\in\R}\|\nabla A(t,\cdot)\|_{W^{1,\infty}(\R^3)}<\infty,
            \\
            \label{def:classA0}
            &\|\nabla A(\cdot,x_0)\|_{\infty}<\frac{\pi}{2\sqrt{3}T},\quad\text{for some }x_0\in\Upsilon(A).
		\end{align}
		Then \eqref{eq:LFE-intro} admits at least two non-constant periodic solutions.
	\end{theorem}

	\begin{remark}\label{remark:magnetostatic}
		From a physical point of view, we are considering smooth electromagnetic fields that are periodic in time and decay at infinity in the space variable. The assumption $\partial_tA\not\equiv0$ rules out the magnetostatic fields
		\begin{equation}\label{def:magnetostatic}
			B(t,x)=B(x),\quad\text{and}\quad E(t,x)\equiv0,\quad\text{for all }(t,x)\in\R^4.
		\end{equation}
		In that regime, the modulus $|\dot q(t)|$ is a first integral of the dynamics, which reduces Equation \eqref{eq:LFE-intro} to a second-order Newton–Lorentz equation, thereby losing its relativistic character. In addition, \eqref{def:classA0} provides a quantitative estimate of the magnetic field at the equilibrium point.  Mathematically, the field given in \eqref{Potential-Fields} is the unique solution for Maxwell's equations under condition \eqref{cond:A-decay} (in the distributional sense, see \cite{GM,GT2} for further details), while the periodicity of the field is necessary for the existence of periodic solutions.
	\end{remark}
	
	The rest of the section is devoted to the verification of the hypotheses of Theorem~\ref{MainTheorem} in the setting of Theorem~\ref{MainTheorem:LFE}. This is a nontrivial task and needs to be carried out in several steps.  In this context, condition \eqref{hyp:convex} requires higher regularity of the vector potential $A(t,x)$ than that assumed in \cite[Theorem 2.1]{GLM}, as well as the global bound \eqref{hyp:bound-secondderivatives} for the second-order spatial derivatives of $A$.
	
	\subsection{Functional framework}\label{Sect:LFE-Functional}
	Following the notation of the abstract setting from Section~\ref{Sect:FunctionalFramework}, we define the Hilbert space $\Hilb$ by
	\begin{equation}\label{Hilbert}
		\Hilb:=\{q\in H^1_{\text{loc}}(\R;\R^3):\,q\text{ is }T\text{-periodic}\},
	\end{equation}
	endowed with the usual inner product and norm of the Sobolev space $H^1$. Moreover, $\Hilb$ admits the canonical decomposition $\Hilb=\Hilb_1\oplus\Hilb_2$, where $\Hilb_1$ is naturally identified with $\R^3$ as the set of constant functions, and $\Hilb_2$ denotes the subspace of zero-mean functions. In this context, we define
	\begin{equation}\label{def:sets-LFE}
		\Ebanach:=L^\infty([0,T];\R^3),\quad \boK:=\{q\in\Hilb:\ \|\dq\|_\infty\leq1\},
	\end{equation}
	and, analogously to \eqref{eq:topology-K}, we introduce on $\boK$ the topology $\tau$ given by
	\begin{equation}\label{def:topology-LFE}
		q_n\to^\tau q\Longleftrightarrow\left\{\begin{array}{ll}
			q_n\to q,&\text{strongly in }\Ebanach,
			\\
			\dot{q}_n\rightharpoonup \dot{q},&\text{in the weak$^*$ topology }\sigma(L^\infty,L^1).
		\end{array}\right.
	\end{equation}
	Consequently, the embedding condition \eqref{eq:embedding} holds and, by Lemma~\ref{appendix-lemma:tau-domain}, $\boK$ is a $\tau$-domain.  In the remainder of this section, the electric field $E(t,x)$ does not appear, so the abuse of notation in \eqref{def:sets-LFE} causes no ambiguity.
	
	For any $A\in\mathcal{C}_T^1\left(\R^4;\R^3\right),$ the action functional  $\boI:\Hilb\to (-\infty,\infty]$ associated to \eqref{eq:LFE-intro} is defined by
	\begin{equation}\label{eq:functional}
		\boI=\Psi+\boF,\quad\Psi(q)=\left\{\begin{array}{ll}
			\displaystyle\int_0^T \left(1-\sqrt{1-|\dq|^2}\right)dt,& q\in\boK,\\
			\infty,&  q\in\Hilb\setminus\boK,
		\end{array}\right.\quad
		\boF(q)=\int_0^T \dq\cdot A(t,q)dt,\ q\in\Hilb.
	\end{equation}
	
	The next result shows that \eqref{eq:functional} fits into the abstract framework $\X(\boK)$ introduced in Section~\ref{Sect:FunctionalFramework}, here established in the setting of \eqref{Hilbert}, \eqref{def:sets-LFE} and \eqref{def:topology-LFE}. 
	\begin{lemma}\label{Lemma:Set-X-LFE}
		For every $A\in\mathcal{C}_T^1(\R^4;\R^3)$, the functional \eqref{eq:functional} belongs to $\X(\boK)$. Furthermore, for every $q\in\boK$, the extension  $\boF'(q): \Ebanach\to \R$ is well-defined as a continuous linear functional, and satisfies
		\begin{equation}\label{eq:F'estimate-LFE}
			\left|\boF'(q)[\varphi]\right|\leq {\hbox{M}}(q)\|\varphi\|_\infty ,\quad\hbox{for all }q\in\boK, \varphi\in \Ebanach,
		\end{equation}
		where $\hbox{M}(q):=\max\left\{\|\partial_jA(\cdot,q(\cdot))\|_{1}:j\in\{t,x_1,x_2,x_3\}\right\},$  and 
		\begin{equation}\label{eq:conv-F'-LFE}
			\boF'(q_n)[\varphi_n]\to\boF'(q)[\varphi],\quad\text{as }n\to\infty,
		\end{equation}
		for every $\{q_n\}\subset\boK$ and $\{\varphi_n\}\subset \Ebanach$ satisfying $q_n\rightharpoonup^\tau q$ and $\varphi_n\to\varphi$ in $\Ebanach$ for some $q\in\boK,\varphi\in \Ebanach$.
	\end{lemma}
	\begin{proof}
It is immediate that $\Psi:\Hilb \to (-\infty,\infty]$ is a convex, proper functional with closed domain $\boK \subset \Hilb$, such that $\Psi|_\boK$ is continuous with respect to the norm $\|\cdot\|_\Hilb$. Moreover, \cite[Proposition 1]{BerJebMaw-Neumann}  directly implies that $\Psi|_\boK$ is lower semicontinuous in $\tau$. Hence $\Psi$ satisfies \ref{H1:nonsmooth}.
		
For any $A\in\mathcal{C}^1(\R^4;\R^3)$, the continuity of the potential, together with the embedding $\Hilb\subset \Ebanach$, implies that the map $t\mapsto A(t,q(t))$ is bounded for every $q\in\Hilb$. In particular, Hölder's inequality yields $|\boF(q)|\leq\|\dq\|_2 \|A(\cdot,q)\|_2,$ for all $q\in\Hilb$, so $\boF:\Hilb\to\R$ is well defined. Moreover, it is clear that $\boF$ is continuous on $\Hilb$.
		
		For convenience of notation, we introduce the function $\boE:\R\times\R^3\times\R^3\to\R^3$ given by 
		\begin{equation}\label{def:epsF'}
			\boE(t,x,y)=\Big(y\cdot \partial_{x_1}A(t,x),  y\cdot\partial_{x_2} A(t,x), y\cdot \partial_{x_3} A(t,x)\Big).
		\end{equation}
		Through this, the derivatives of $\boF$ are the linear operators $\boF'(q):\Hilb\to\R$ given by
		\begin{equation}\label{eq:F'Hilbert}
			\boF'(q)[\varphi]=\int_0^T \left[\dot{\varphi}\cdot A(t,q)+\varphi\cdot\boE(t,q,\dq)\right]dt,\quad\hbox{for all }q,\varphi\in\Hilb.
		\end{equation}
		It is clear that the map $\boF':\Hilb\to \Hilb$, $q\mapsto \boF'(q)$, is well defined and continuous, where each image $\boF'(q)$ is identified with an element of $\Hilb$ via the Riesz representation theorem. Moreover, integrating by parts yields
			\begin{equation}\label{eq:F'-extended}
            \boF'(q)[\varphi]=\int_0^T\varphi\cdot\left(-\frac{d}{dt}(A(t,q(t)))+\boE(t,q,\dq)\right)dt,
            \end{equation}
			so the extension  $\boF'(q):\Ebanach\to\R$ is well-defined as a  continuous linear functional, and 
			 \eqref{eq:F'estimate-LFE} trivially holds. From this, estimate \eqref{hyp:F-dualE} follows directly. Consequently, $\boF:\Hilb\to\R$ satisfies \ref{H2:smooth}. Therefore, $\boI\in\X(\boK)$ for every $A\in\mathcal{C}^1_T(\R^4;\R^3)$.  Finally, bearing in mind \eqref{eq:F'-extended},  \eqref{eq:conv-F'-LFE} follows by standard arguments.
	\end{proof}
	
	The next result identifies a class of potentials for which the action functional \eqref{eq:functional} satisfies $\PS$ with respect to the topology \eqref{def:topology-LFE}.
	
	\begin{lemma}\label{lemma:PS2}
		For every $A\in\mathcal{C}_T^1(\R^4;\R^3)$ satisfying \eqref{cond:A-decay}, the functional \eqref{eq:functional} satisfies $\PS$.
	\end{lemma}

	\begin{proof}
		Let $\{q_n\}\subset\boK$ be a Palais--Smale sequence with level $c\in\R$, bounded in $\Ebanach$. Since $\boK$ is a $\tau$-domain, there exists $q\in\boK$ such that, up to a subsequence, $q_n\to^\tau q$. Moreover, $\boI|_\boK$ is lower semicontinuous with respect to $\tau$, hence
		\begin{equation}
			\boI(q)\leq\liminf_{n\to\infty}\boI(q_n)=c.
		\end{equation}
		To prove the reverse inequality, observe first that, since $\{q_n\}$ is bounded in $\Ebanach$ and $q_n,q\in\boK$, it follows
			\[\sup_{n}\|q-q_n\|^2_\Hilb\leq T(2+\sup_n\|q-q_n\|_\infty^2)<\infty.\]
			Thus, taking $p=q$ in \eqref{eq:PalaisSmale} and passing to the limit superior as $n\to\infty$, using \eqref{eq:F'estimate-LFE}, we obtain 
			\begin{equation}
				\Psi(q)\geq\limsup_{n\to\infty}\Psi(q_n).
			\end{equation}
			This, combined with the lower semicontinuity of $\Psi$ in $\tau$, implies that $\Psi(q_n)\to\Psi(q)$ as $n\to\infty$. Therefore, since $\boF|_\boK$ is continuous with respect to $\tau$, we conclude
			\begin{equation*}
				\boI(q)=\lim_{n\to\infty}\boI(q_n)=c.
			\end{equation*}
		Finally, passing to the limit by \eqref{eq:conv-F'-LFE} in \eqref{eq:PalaisSmale} for arbitrary $p\in\boK$, we deduce that $q$ is a critical point of $\boI$ at level $c$. In particular, $q$ is non-constant when $c\neq0$, since $\boI|_{\R^3}\equiv0$.
		
		Hence, the result follows once we prove that every Palais--Smale sequence at a nonzero level is bounded in $\Ebanach$. To this end, \cite[Lemma 3.4]{GLM} states that every $A\in\mathcal{C}_T^1(\R^4;\R^3)$ satisfying \eqref{cond:A-decay} is such that
		\begin{equation}\label{statement:vanishingA}
			\lim_{n\to\infty}\int_0^T \dq_n\cdot A(t,q_n)dt= 0,\quad\hbox{for all }\{q_n\}\subset\boK\hbox{ with }\lim_{n\to\infty}|\bar{q}_n|=\infty,
		\end{equation}
		where $\bar{q}=T^{-1}\int_0^Tq(t)dt$. Analogously to \cite[Lemma 3.4]{GLM}, it is straightforward to show that, for every $A\in\mathcal{C}_T^1(\R^4;\R^3)$ satisfying \eqref{cond:A-decay}, the following holds:
		\begin{equation}\label{statement:vanishingE}
			\lim_{n\to\infty}\int_0^T  \tilde q_n\cdot\boE(t,q_n,\dq_n) dt= 0,\quad\hbox{for all }\{q_n\}\subset\boK\hbox{ with }\lim_{n\to\infty}|\bar{q}_n|=\infty,
		\end{equation}
		where $\boE$ is defined in \eqref{def:epsF'} and $\tilde q:=q-\bar{q}$, for all $q\in\Hilb$.
		
		Now let $\{q_n\}\subset\boK$ be a Palais–Smale sequence with level $c\not=0$, let us assume it is unbounded in the $\Ebanach$ norm, and let us extract a not relabeled subsequence such that $|\bar q_n|\to\infty$ as $n\to\infty$. On the one hand, since
        \begin{equation}\label{ineq:functional-force}
				\boI(q_n)\geq\int_0^T \dq_n\cdot  A(t,q_n)dt,
			\end{equation}
		passing to the limits as $n\to\infty$, together with \eqref{statement:vanishingA}, we obtain $c\geq 0$. 
		
		On the other hand, since $\Psi|_{\R^3}\equiv0$, and using the expression for $\boF'$ given in \eqref{eq:F'Hilbert}, the inequality \eqref{eq:PalaisSmale} with $p=\bar q_n$ becomes
		\begin{equation}
			-\Psi(q_n) - \int_0^T \dq_n\cdot A(t,q_n) dt - \int_0^T \tilde q_n\cdot\boE(t,q_n,\dq_n) dt=-\boI(q_n) - \int_0^T \tilde q_n\cdot\boE(t,q_n,\dq_n) dt \geq -\epsilon_n\|\tilde q_n\|_{\Hilb}.
		\end{equation}
		Moreover, bearing in mind that $\|\tilde q\|_\infty\leq\sqrt{3}T$ for all $q\in\boK$, we obtain
		\[\sup_{n}\|\tilde q_n\|^2_\Hilb\leq T(2+\sup_n\|\tilde q_n\|_\infty^2)\leq T(2+3T^2),\]
			and, as a consequence,
			\begin{equation*}
				-\boI(q_n) - \int_0^T \tilde q_n\cdot\boE(t,q_n,\dq_n)dt \geq -\epsilon_n\sqrt{T(2+3T^2)},\quad\hbox{for all }n\in\mathbb{N}.
			\end{equation*}
		Finally, passing to the limit in the previous inequality, and applying \eqref{statement:vanishingE}, yields $-c\geq 0$. Therefore $c=0$ necessarily, contradicting the assumption  $c\neq0$. Hence $\{q_n\}$ is bounded in $\Ebanach$, which completes the proof.
	\end{proof}

	\begin{remark}\label{Remark:PS-0}
		This result is analogous to \cite[Lemma 3.5]{AS}, which also applies to singular scalar potentials (although our argument here is also valid for the pairs $(\Phi,A)$ considered in \cite{GLM}, allowing for the presence of singularities). The novelty of Lemma~\ref{lemma:PS2} lies in the fact that it does not require any asymptotic behavior of the vector potential $A(t,x)$ at infinity, but only of its derivatives. We also stress that the condition $c\neq0$ in Lemma~\ref{lemma:PS2} cannot be removed in general. To see this, consider any sequence $\{q_n\}\subset\R^3$. In that case, it is immediate that $\boI(q_n)=0$ for all $n\in\mathbb{N}$, while Equation \eqref{eq:F'Hilbert} reduces to
		\begin{equation}
			\boF'(q_n)[\varphi]=\int_0^T \varphi\cdot\partial_t A(t,q_n)\, dt,\quad\hbox{for all }\varphi\in\Hilb.
		\end{equation}
		Assuming that $|q_n|\to\infty$ as $n\to\infty$, by \eqref{cond:A-decay} it follows that
		$\boF'(q_n)[\varphi]\to 0$ as $n\to\infty$, for all $\varphi\in\Hilb$. It is clear that $\{q_n\}$ satisfies inequality \eqref{eq:PalaisSmale}, hence it is an unbounded Palais–Smale sequence for the functional \eqref{eq:functional} at level zero.
	\end{remark}
	\begin{remark}
		Following the proof Lemma~\ref{lemma:PS2}, one infers that the extra conditions
		\begin{equation}\label{hyp:decay}
			 \lim_{n\to\infty}\boI(q_n)=0,\quad\text{for every Palais--Smale sequence }{q_n}\subset\boK\text{ with }\lim_{n\to\infty}\|q_n\|_{\Ebanach}=\infty,
		\end{equation}
		and	\begin{equation}\label{hyp:FContinuousTau}
			\hbox{$\boF'|_\boK:\boK\to \Ebanach^*$ is continuous in $\tau$},
		\end{equation}
		are sufficient for $\PS$ to hold even in the abstract setting of section~\ref{Sect:FunctionalFramework}.
	\end{remark}
	Finally, we provide sufficient conditions on the vector potential under which hypothesis \eqref{hyp:convex} holds, allowing the regularization procedure described in Section~\ref{Sect:eke-las} to apply to the action functional \eqref{eq:functional}.
	\begin{lemma}\label{lemma:FunctionalC2}
		Let $A\in\mathcal{C}_T^2(\R^4;\R^3)$ satisfy \eqref{cond:A-decay} and \eqref{hyp:bound-secondderivatives}. Then, $\boF:\Hilb\to\R$ is of class $\boC^2$ and there exists a constant $C>0$  such that
		\begin{equation}\label{eq:F''-estimate}
			|\boF''(q)[\varphi,\varphi]|\leq C\|\varphi\|_{\Hilb}^2,\quad\text{for all }(q,\varphi)\in \boK\times\Hilb.
		\end{equation}
		In particular,  \eqref{hyp:convex} holds for every $\mu\geq C/2$.
	\end{lemma}
	\begin{proof}
		It is straightforward to verify that
		\begin{align}
			\boF''(q)[\psi,\varphi]=\sum_{j,k}&\int_0^T\left[\dot\varphi_j\psi_k\partial_{x_k}A_j(t,q)+\varphi_j\dot\psi_k\partial_{x_j}A_k(t,q) +\sum_{l}\int_0^T\varphi_j\dq_k\psi_l\partial^2_{x_jx_l}A_k(t,q)\right]dt,
		\end{align}
		for all $q\in\boK$, and $\varphi,\psi\in\Hilb$.
		Note that by \eqref{cond:A-decay} and \eqref{hyp:bound-secondderivatives}, these derivatives of the vector potential are uniformly bounded in $\R^4$. Hence, denoting by $\mathcal{B}(\Hilb)$ the space of bilinear maps from $\Hilb\times\Hilb$ to $\R$, arguments analogous to those in the proof of Lemma~\ref{Lemma:Set-X-LFE} show that $\boF'':\Hilb\to\mathcal{B}(\Hilb)$ is well defined and continuous. Moreover, using H\"older inequality, \eqref{cond:A-decay}, and \eqref{hyp:bound-secondderivatives}, the estimate \eqref{eq:F''-estimate} follows directly. As a consequence, for all $\mu\geq C$, it follows that
		\begin{equation}
			\boF''(q)+\mu\,\mathrm{Id}:\Hilb\times\Hilb\to\R \text{ is positive definite for all } q\in\boK,
		\end{equation}
		where $\mathrm{Id}[\psi,\varphi]:=\langle\psi,\varphi\rangle_\Hilb$. Using that $\boK$ is convex, it is straightforward to adapt the arguments in \cite[p.~307]{Eke-Las} to derive that $\boI+\mu\|\cdot\|_\Hilb^2$ is convex on $\boK$ for every $\mu\geq C/2$. Moreover, since $\boI(q)=\infty$ for $q\notin\boK$, the functional $\boI+\mu\|\cdot\|_\Hilb^2$ extends trivially as a convex functional on $\Hilb\setminus\boK$. Hence, hypothesis \eqref{hyp:convex} holds for all $\mu\geq C/2$, which completes the proof.
	\end{proof}

    \subsection{Multiplicity of critical points }\label{Sect-LFE-Final}
	
	In this section we prove Theorem~\ref{MainTheorem:LFE} by combining the variational framework developed in \cite{GLM} with the abstract multiplicity result established in Theorem~\ref{MainTheorem}.
	
	 The variational formulation of the Lorentz force equation is typically established on the space $\boW$ of $T$-periodic Lipschitz functions. In that setting, \cite[Theorem 6]{ABT} characterizes the periodic solutions of \eqref{eq:LFE-intro} as critical points of \eqref{eq:functional}. Using the compact embedding $\boW\subset\Hilb$, this characterization extends with the same proof to the Hilbert space framework from Section~\ref{Sect:LFE-Functional}. The result can be stated as follows:
	\begin{theorem}\label{theorem:CriticalPoint}
		A function $q\in\boK$ is a solution of \eqref{eq:LFE-intro} if, and only if, $q$ is a critical point of \eqref{eq:functional}.
	\end{theorem} 
	
	Following the approach in \cite{GLM}, we denote by $\boA_0$ the set of vector potentials satisfying the hypotheses of Theorem~\ref{MainTheorem:LFE}, namely
	\begin{equation}\label{def:FamilyA0}
		\boA_0:=\{A\in\boC_T^2(\R^4;\R^3) \hbox{ satisfying $\partial_t A\not\equiv 0$, \eqref{cond:A-decay}, \eqref{hyp:bound-secondderivatives}, and \eqref{def:classA0}}\}.
	\end{equation}
	In particular, for every $A\in\boA_0$, the action functional \eqref{eq:functional} admits a global minimizer at a negative level, since $\boA_0$ is contained in the class $\boA$ introduced in \cite[Theorem 2.1]{GLM}.
	
	\begin{definition}
		The set of \textsl{normalized potentials} in $\boA_0$ is defined by
		\begin{equation}
			\boA_0^N:=\{A(t,x+x_0)-A(t,x_0):\ A\in\boA_0, \hbox{ and
			}  x_0\in\Upsilon(A)\}.
		\end{equation}
	\end{definition}
	\begin{remark}\label{rem:equivalenceCP}
		Equivalently, $\boA_0^N$ consists of those potentials $A'\in\boA_0$ for which the origin is an isolated equilibrium ($0\in\Upsilon(A')$), condition \eqref{def:classA0} is satisfied at $0$, and $A'(\cdot,0)\equiv0$. We emphasize that the number of critical points of \eqref{eq:functional} is invariant under the normalization procedure defining $\boA_0^N$. Indeed, given $A\in\boA_0$, let $A'\in\boA_0^N$ be an associated normalized potential, namely
		\begin{equation}
			A'(t,x)=A(t,x+x_0)-A(t,x_0),\quad\hbox{for some }x_0\in\Upsilon(A).
		\end{equation}
		Then, it is straightforward to verify that $q\in\boK$ solves \eqref{eq:LFE-intro} for $A$ if, and only if, $q-x_0$ solves \eqref{eq:LFE-intro} for $A'$. Consequently, the functionals \eqref{eq:functional} associated with $A$ and $A'$ have the same number of critical points, including both constant and nontrivial ones.
	\end{remark}
	
	Given $A\in\boA_0^N$, by the continuity of $\nabla A$, together with \eqref{def:classA0} at $x_0=0$, there exists $r>0$ such that
	\begin{equation}\label{def:r}
		\|\nabla A(\cdot,x)||_{\infty}<\frac{\pi}{2\sqrt{3}T},\ \hbox{ for all }|x|\leq c_0 r,
	\end{equation}
	where $c_0$ is the constant in \eqref{eq:embedding}. Then, defining
    $$
	\ell:=2\sqrt{3}\max_{|x|\leq c_0r}\|\nabla A(\cdot,x)||_{\infty}<\frac{\pi}{T},$$
    and using that $A(\cdot,0)\equiv0$ together with the mean value theorem, we obtain 
	\begin{equation}\label{hyp:linear}
		|A(t,x)|\leq \frac{\ell}{2}|x|,\quad\text{for all }t\in[0,T] \hbox{ and all }|x|\leq c_0r.
	\end{equation}
	
	Recalling that $\boI|_{\R^3}\equiv0$ trivially, for all $A\in\mathcal{C}^1_T(\R^4;\R^3)$, the next result establishes the local linking \eqref{hyp:Splitting} for the normalized potentials. 
	
	\begin{lemma}\label{lemma:splitting}
		For any $A\in\boA_0^N$, the functional $\boI$ defined in \eqref{eq:functional} satisfies 
		\begin{equation}\label{statement:X}
			\boI(q)\geq C \|q\|_\Hilb^2, \quad\text{for every }q\in\Hilb_2\text{ with }\|q\|_\Hilb\in(0,r],
		\end{equation}
		where $C:=\dfrac{\pi(\pi-T\ell)}{2(\pi^2+T^2)}$, and the constants $r$, $\ell$ are given in \eqref{hyp:linear}.
	\end{lemma}
	\begin{proof}
		Let $A\in\mathcal{A}_0^N$ be fixed.
		Observe that \eqref{statement:X} is trivially satisfied when $q\notin\boK$, hence it suffices to consider the case $q\in\boK\cap\Hilb_2$ with $\|q\|_\Hilb\in(0,r]$, so that $\|q\|_\infty\in (0,c_0r]$. Since $\Psi(q)\geq \frac12\|\dq\|_2^2$ for all $q\in\boK$, by Hölder’s inequality and \eqref{hyp:linear} we obtain
		\begin{align}
			\boI(q)\geq\frac12\|\dq\|_2^2 - \|A(t,q)\|_2 \|\dq\|_2\geq \frac12\|\dq\|_2\left(\|\dq\|_2 - \ell \|q\|_2\right),\quad\hbox{for all }q\in\boK\hbox{ with }\|q\|_\Hilb\in(0,r].
		\end{align}
		Additionally, the Poincaré-Wirtinger inequality states that $\pi    \|q\|_2\leq T \|\dq\|_2$ for all $q\in \Hilb_2$,
		which implies that
		$\|q\|^2_\Hilb\leq(T^2\pi^{-2}+1) \|\dq\|^2_2$, for all $q\in\Hilb_2$. Therefore, 
		\begin{align*}
			\boI(q)&\geq \frac12\left(1-\frac{\ell T}{\pi}\right)\|\dq\|_2^2\geq C \|q\|_\Hilb^2 ,\quad\hbox{for all }q\in\Hilb_2\cap\boK\hbox{ with }\|q\|_\Hilb\in(0,r],
		\end{align*} 
		which completes the proof. 
	\end{proof}
	We are now in a position to prove the main result of this section. The proof follows from a direct application of Theorem~\ref{MainTheorem} together with the lemmas established above.

\begin{proof}[Proof of Theorem~\ref{MainTheorem:LFE}]
		By Remark~\ref{rem:equivalenceCP}, it suffices to prove the result in the class $\mathcal{A}^N_0$. 
        
        Let us fix $A\in\mathcal{A}^N_0$. Then, by \cite[Theorem 2.1]{GLM} there exists $\minq \in \boK\setminus\R^3$ such that $\min_{\Hilb}\boI=\boI(\minq)<0$. Moreover, Lemmas~\ref{Lemma:Set-X-LFE},  \ref{lemma:PS2}, \ref{lemma:FunctionalC2}, \ref{lemma:splitting} guarantee that the hypotheses of Theorem~\ref{MainTheorem} are satisfied. Therefore, there exists a second critical point $q\in\boK\setminus\R^3$ at a nonzero level and, by Theorem \ref{theorem:CriticalPoint}, $q$ is a periodic solution of \eqref{eq:LFE-intro}.
        \end{proof}

\section{The mean curvature operator in Minkowski spaces}
\label{Section:Minkowski}

In this section we analyze the Dirichlet problem \eqref{eq:PDE-Minkowski} for a bounded domain
 $\Omega\subset\mathbb{R}^N$ ($N\geq 1$)  with boundary
$\partial\Omega$ of class $\mathcal{C}^2$, and $f\colon\Omega\times\mathbb{R}\to\mathbb{R}$ a given function. For a pointwise solution to~\eqref{eq:PDE-Minkowski} to be well defined, its
gradient must remain uniformly bounded away from the singularity of the mean
curvature operator.
To this end, we adopt the following notion of solutions:

\begin{definition}\label{def:sol-PDE}
A \emph{solution} of~\eqref{eq:PDE-Minkowski} is a function
$q \in W^{2,p}(\Omega)$, for some $p > d$, satisfying
$\|\nabla q\|_{\infty} < 1$ and $q|_{\partial\Omega}=0$, such that
\eqref{eq:PDE-Minkowski} holds almost everywhere in~$\Omega$.
\end{definition}

Our starting point is the minimization result of~\cite{BerJebMaw}.
Specifically, that article assumes that $f$ is a Carath\'{e}odory function,
meaning
\begin{equation}\label{hyp:Cath}
  f(x,\cdot)\colon\mathbb{R}\to\mathbb{R}
    \ \text{ is continuous for a.e.\ } x\in\Omega,
  \qquad
  f(\cdot,s)\colon\Omega\to\mathbb{R}
    \ \text{ is measurable for all } s\in\mathbb{R}.
\end{equation}
In addition, for every $\rho>0$, there exists $\alpha_\rho\in L^\infty(\Omega)$,
$\alpha_\rho > 0$, such that
\begin{equation}\label{hyp:growth-condition}
  |f(x,s)|\leq \alpha_\rho(x)
  \quad\text{for all } s\in[-\rho,\rho]
  \text{ and for a.e. } x\in\Omega.
\end{equation}
Under hypotheses~\eqref{hyp:Cath} and~\eqref{hyp:growth-condition},
\cite[Theorem~2.1]{BerJebMaw} asserts that the Dirichlet
problem~\eqref{eq:PDE-Minkowski} admits at least one solution, which is a
global minimizer of the associated action functional.
As observed in~\cite[Theorem~3.1]{BerJebMaw}, an additional condition guaranteeing that the
functional attains negative values ensures that such minimizers are nontrivial.

In the present section we consider a broad family of nonlinearities within the
framework of~\cite{BerJebMaw} for which Theorem~\ref{MainTheorem} yields a
second nontrivial solution of~\eqref{eq:PDE-Minkowski}.
More precisely, we establish the following result.

\begin{theorem}\label{MainTheorem-Dirichlet}
Let $f\colon\Omega\times\mathbb{R}\to\mathbb{R}$ be of the form
\begin{equation}\label{hyp:PDE-f0-multiplicity}
  f(x,s)=a(x)s - g(s),
\end{equation}
where $a\in L^\infty(\Omega)$ and $g\in\mathcal{C}^1(\mathbb{R})$ satisfies
\begin{equation}\label{hyp:PDE-f1-multiplicity}
  g(0)=g'(0)=0.
\end{equation}
Suppose moreover that the operator $-\Delta+a$ on $H_0^1(\Omega)$ has a
negative first eigenvalue and no zero eigenvalue.
Then \eqref{eq:PDE-Minkowski} admits at least two nontrivial solutions.
\end{theorem}

This theorem is a relativistic counterpart of~\cite[Theorem~6]{Bre-Nir}. Note
that the datum~\eqref{hyp:PDE-f0-multiplicity} trivially
satisfies both~\eqref{hyp:Cath} and~\eqref{hyp:growth-condition}, so \cite[Theorem~2.1]{BerJebMaw} implies the existence of an action minimizer. However, under \eqref{hyp:PDE-f0-multiplicity} and \eqref{hyp:PDE-f1-multiplicity}, zero is trivially a solution to \eqref{eq:PDE-Minkowski}. We will show that the action functional attains negative values (and hence the minimizer is nontrivial) following an argument which, in contrast to  \cite[Theorem~3.1]{BerJebMaw} or \cite[Theorem~6]{Bre-Nir}, does not assume any size condition on $g$.

The proof of Theorem~\ref{MainTheorem-Dirichlet} is given at the end of
Section~\ref{Sect:Mink-Dirichlet}.
In preparation, Section~\ref{Section:Mink-Functional} reformulates
\eqref{eq:PDE-Minkowski} in a variational framework in which the hypotheses of
Theorem~\ref{MainTheorem} are verified.
\subsection{Functional framework}\label{Section:Mink-Functional}
In what follows, we take $\Hilb=H_0^1(\Omega)$ and
\begin{equation}\label{eq:sets-PDE}
\Ebanach:=L^p(\Omega)\quad\hbox{with}\quad \left\{\begin{array}{cl}
         
          p\in(2,\infty)&\hbox{if }                 N\in\{1,2\}, \\
         p=\dfrac{2N}{N-2}&\hbox{if }N\geq3, 
    \end{array}\right.\qquad\boK:=\{u\in H_0^1(\Omega):\ \|\nabla u\|_\infty\leq1\}.
\end{equation}
With these choices, the embedding condition \eqref{eq:embedding} holds trivially. By Lemma~\ref{appendix-lemma:tau-domain} in Appendix, the set $\boK$ is a $\tau$-domain in the topology defined in \eqref{eq:topology-K}. 

The next result yields a universal $L^\infty$ estimate and a compactness property in $\boK$ which are specific to the Dirichlet setting.

\begin{lemma}\label{appendix-lemma:convergence-PDE}
The following estimate holds:
\begin{equation}\label{eq:universal-estimate}
    \|q\|_\infty\leq\operatorname{diam}(\Omega),\quad\text{for all }q\in\boK.
\end{equation}
Moreover, the set $\boK$ is compact in the topology $\tau$ given in \eqref{eq:topology-K}. 
\end{lemma}
\begin{proof}
Let $q\in\boK$. By Lemma~\ref{appendix-lemma:tau-domain}, $q\in W^{1,\infty}(\Omega)$ and, in particular, $q\in\boC(\overline\Omega)$. For any $x\in\Omega$, let us consider a point $y(x)\in\partial\Omega$ such that $|x-y(x)|=\min\{|x-y|:\ y\in\partial\Omega\}$. In particular, the segment joining $x$ to $y(x)$ is contained in $\Omega$ and $q(y(x))=0$. Therefore, the mean value theorem leads to
\[|q(x)|\leq |x-y(x)|\leq\operatorname{diam}(\Omega),\quad\text{for all }x\in\Omega,\]
which readily implies \eqref{eq:universal-estimate}. In particular, every sequence $\{q_n\}\subset\boK$ is bounded in $L^\infty(\Omega)$, and hence in the Banach space $\Ebanach$. Since $\boK$ is a $\tau$-domain, $\{q_n\}$ admits a subsequence converging in $\tau$ to some $q\in\boK$, which completes the proof.
\end{proof}	

Taking \eqref{eq:universal-estimate} into account, by applying a suitable $\boC^1$ truncation of $f$ for values $|s|\geq\operatorname{diam}(\Omega)$, we assume from now on, without loss of generality, that there exists a constant $\beta>0$ such that 
\begin{equation}\label{eq:truncation}
    |f(x,s)|\leq \beta,\quad\text{for a.e. } x\in \Omega, \text{ for all }s\in\R.
\end{equation}
The action functional $\boI:\Hilb\to(-\infty,\infty]$ associated with \eqref{eq:PDE-Minkowski} is defined by
\begin{equation}\label{eq:functional-PDE}
    \boI=\Psi+\boF,\quad \Psi(q)=\left\{\begin{array}{ll}\displaystyle\int_\Omega\left(1-\sqrt{1-|\nabla q|^2}\right) dx,&q\in\boK,\\
        \infty,&  q\in\Hilb\setminus\boK,
   \end{array}\right.\quad\boF(q)=\int_\Omega F(x,q)dx,\,q\in\Hilb,
\end{equation}
where $F(x,s)=\int_0^s f(x,t)dt$. We keep the notation $\X(\boK)$ as introduced in Section~\ref{Sect:FunctionalFramework}, now in the framework of \eqref{eq:sets-PDE}. The following result is analogous to Lemma~\ref{Lemma:Set-X-LFE}.
\begin{lemma}\label{lemma:functional-PDE}
For every $f:\Omega\times\R\to\R$ satisfying \eqref{hyp:Cath} and \eqref{eq:truncation}, the functional \eqref{eq:functional-PDE} belongs to $\X(\boK)$.
\end{lemma}
\begin{proof}
Bearing in mind \eqref{eq:truncation}, it is standard to verify that $\boF\in\boC^1(\Hilb;\R)$. Moreover, for an arbitrary $\varepsilon>0$, \eqref{eq:truncation}, together with the mean value theorem and Hölder's inequality, implies
\begin{equation}
    |\boF(q_1)-\boF(q_2)|\leq\beta \int_\Omega|q_1(x)-q_2(x)|dx\leq\varepsilon ,\quad\hbox{for all }q_1,q_2\in\boK \text{ with }\|q_1-q_2\|_p\leq\delta:=\varepsilon\beta^{-1}|\Omega|^{-1/p'}.
\end{equation}
In particular, $\boF$ satisfies \ref{H2:smooth}. Moreover, as remarked in the proof of Lemma~\ref{Lemma:Set-X-LFE}, it is standard to verify that $\Psi:\Hilb\to(-\infty,\infty]$ satisfies \ref{H1:nonsmooth}. This completes the proof.
\end{proof}

The properties established in Lemma~\ref{appendix-lemma:convergence-PDE} and Lemma~\ref{lemma:functional-PDE} yield the next result.
\begin{lemma}\label{lemma:PS-PDE}
For every $f:\Omega\times\mathbb{R}\to\mathbb{R}$ satisfying \eqref{hyp:Cath} and \eqref{eq:truncation}, the functional \eqref{eq:functional-PDE} satisfies $\PS$.
\end{lemma}
\begin{proof}
Let $f$ satisfy  \eqref{hyp:Cath} and \eqref{eq:truncation}, and let $\{q_n\}\subset\boK$ be a Palais--Smale sequence. By Lemma~\ref{appendix-lemma:convergence-PDE}, up to a subsequence, $q_n\to^\tau q$ for some $q\in\boK$. Let us show that $q$ is a critical point of $\boI$. To do so, for $\varphi\in\Hilb$ one has 
\begin{equation}\label{eq:F'-PDE}
    \boF'(q_n)[\varphi-q_n]=\int_\Omega f(x,q_n)(\varphi-q_n )dx=\int_\Omega f(x,q_n)(\varphi-q)dx+\int_\Omega f(x,q_n)(q-q_n )dx.
\end{equation}
Using \eqref{eq:truncation}, the continuity of $f(x,\cdot)$, and H\"older's inequality, we obtain
\begin{equation}
    \left|\int_\Omega f(x,q_n)(q-q_n)dx\right|\leq\beta |\Omega|^{1/{p'}}\|q-q_n\|_p.
\end{equation}
Therefore, this term vanishes as $n\to\infty$, since $\|q_n-q\|_p\to 0$.

On the other hand, up to a subsequence, $q_n\to q$ a.e. in $\Omega$, and therefore $f(x,q_n(x))\to f(x,q(x))$ for a.e. $x\in\Omega$, by the continuity of $f(x,\cdot)$. Since $\varphi-q\in L^1(\Omega)$ and by \eqref{eq:truncation}, the dominated convergence theorem yields $\int_\Omega f(x,q_n)(\varphi-q)dx\to\int_\Omega f(x,q)(\varphi-q)$ as $n\to\infty$. Thus,
\begin{equation}
    \lim_{n\to\infty}\boF'(q_n)[\varphi-q_n]=\boF'(q)[\varphi-q],\quad\hbox{for all }\varphi\in\Hilb.
\end{equation}
As a consequence, passing to the limit in \eqref{eq:PalaisSmale} as $n\to\infty$, and using the lower semicontinuity of $\Psi$, we obtain
\begin{equation}
    \Psi(\varphi)-\Psi(q)+\boF'(q)[\varphi-q]\geq0,\quad\hbox{for all }\varphi\in \boK,
\end{equation}
which characterizes $q$ as a critical point of \eqref{eq:functional-PDE}. Finally, arguing as in the proof of Lemma~\ref{lemma:PS2}, we conclude that $\boI(q_n)\to\boI(q)$ as $n\to\infty$. This proves that every Palais--Smale sequence admits a subsequence converging in $\tau$ to a critical point at the corresponding level. In particular, $\PS$ holds.
\end{proof}
We conclude this section by proving that, under the hypotheses of Theorem~\ref{MainTheorem-Dirichlet}, \eqref{hyp:convex} holds for \eqref{eq:functional-PDE}. 

\begin{lemma}\label{lemma:FunctionalC2-Dirichlet}
Let $f:\Omega\times\R\to\R$ satisfy \eqref{hyp:PDE-f0-multiplicity} and \eqref{hyp:PDE-f1-multiplicity}. Then, $\boF:\Hilb\to\R$ is of class $\boC^2$ and there exists a constant $C>0$ such that  \begin{equation}\label{bounded2der-PDE-Dirichlet}
    |\boF''(q)[\varphi,\varphi]|\leq C\|\varphi\|_{\Hilb}^2,\quad\text{for all }q\in \boK, \hbox{and all }\varphi\in\Hilb.
    \end{equation}
    In particular, \eqref{hyp:convex} holds for all $\mu\geq C/2$.
\end{lemma}
\begin{proof}
 It is straightforward to verify that $\boF$ is of class $\boC^2$, with
 \begin{equation*}
     \boF''(q)[\psi,\varphi]=\int_\Omega\partial_s f(x,q)\psi\varphi dx=\int_\Omega \left(a(x)- g'(q)\right)\psi\varphi dx,\quad\hbox{for all }(q,\psi,\varphi)\in\Hilb^3.
 \end{equation*}
Moreover, by choosing the truncation conveniently, one has the estimate \eqref{eq:truncation} also for $\partial_s f$, so that
 \begin{equation}
    \left| \boF''(q)[\varphi,\varphi]\right|\leq\beta \|\varphi\|_\Hilb^2,\quad\hbox{for all }(q,\varphi)\in\boK\times\Hilb.
 \end{equation}
From this point, the same argument as in Lemma~\ref{lemma:FunctionalC2} completes the proof.
\end{proof}

\subsection{Multiplicity of critical points}\label{Sect:Mink-Dirichlet}

The next theorem is the counterpart of Theorem~\ref{theorem:CriticalPoint} for Equation \eqref{eq:PDE-Minkowski}, since it characterizes its solutions variationally as critical points of \eqref{eq:functional-PDE}. This result is essentially contained in \cite[Theorem 2.1]{BerJebMaw}. Adapting the proof to our setting is immediate so we omit it.
\begin{theorem}\label{Th:criticalpoint-PDE}
    A function $q\in\boK$ is a solution of \eqref{eq:PDE-Minkowski} if, and only if, $q$ is a critical point of \eqref{eq:functional-PDE}.
\end{theorem}

The verification of the local linking condition \eqref{hyp:Splitting} follows from the classical spectral analysis of the operator $-\Delta+a$ in $\Omega$ under Dirichlet boundary conditions. More precisely, it is well known that its spectrum is a non-decreasing divergent sequence $\{\lambda_n\}\subset\R$. We will assume that
\begin{equation}\label{eq:negative_eigenvalues}
    \lambda_k< 0<\lambda_{k+1},\quad\text{for some }k\in\mathbb{N}.
\end{equation}
Let us denote by $\{\varphi_n\}$ the associated eigenfunctions (unitary in $\|\cdot\|_2)$ forming an orthonormal basis of $L^2(\Omega)$, and let us consider the subspaces
\begin{equation*}    \Hilb_1=\operatorname{span}\{\varphi_1,\dots,\varphi_k\},\quad \Hilb_2=\Big\{u\in \Hilb:\ \int_\Omega uv dx=0,\ \text{for all }v\in\Hilb_1\Big\},
\end{equation*}
so that $\Hilb=\Hilb_1\oplus\Hilb_2$.

\begin{lemma}\label{lemma:splitting-PDE}
    Let $f:\Omega\times\R\to\R$ satisfy \eqref{hyp:PDE-f0-multiplicity} and \eqref{hyp:PDE-f1-multiplicity}, for some $\delta>0$ sufficiently small, and assume that \eqref{eq:negative_eigenvalues} holds. Then, there exist $C,r_0>0$ such that
\begin{equation}\label{hyp:Splitting-PDE}
\left\{\begin{array}{ll}
\boI(q)\leq -C\|q\|_{\Hilb}^2,&\hbox{for all }q\in\Hilb_1\ \hbox{ with }\|q\|_\Hilb\in(0,r],\\
\boI(q)\geq C\|q\|_\Hilb^2,&\hbox{for all }q\in\Hilb_2\ \hbox{ with }\|q\|_\Hilb\in(0,r].
\end{array}\right.
\end{equation}
\end{lemma}

\begin{remark}\label{remark:negative-minimum-PDE}
    Recall that \cite[Theorem~2.1]{BerJebMaw} provides a minimizer of $\boI$. Notice that the first inequality of \eqref{hyp:Splitting-PDE} implies that $\inf_\Hilb\boI<0$ attains negative values near zero. In particular, the minimizers cannot be zero.
\end{remark}

\begin{proof}[Proof of Lemma~\ref{lemma:splitting-PDE}]
In the proof, we will denote $G(s)=\int_0^s g(t)dt$. 

Let $q\in \Hilb_2$. If $q\not\in\boK$, then $\boI(q)\geq C\|q\|_\Hilb^2$ trivially. On the contrary, if $q\in\boK$, then
\[\boI(q)\geq \frac12\int_\Omega\left(|\nabla q|^2+a(x)q^2-2 G(q))\right)dx.\]
Hence, \cite[Lemma 2.15]{Willem} states that there exists $C_1>0$ such that
\[\boI(q)\geq C_1\|q\|_\Hilb^2-\int_\Omega G(q)dx.\]
In addition, by \eqref{hyp:PDE-f1-multiplicity} it follows that $G(s)=o(s^2)$ whenever $s\approx 0$. Moreover,  $G\in L^\infty(\R)$ by \eqref{eq:truncation}. Hence, for every $\varepsilon>0$, there exists $C_\varepsilon>0$ such that
\begin{equation}\label{eq:G-behavior}
    |G(s)|\leq\varepsilon s^2 + C_\varepsilon |s|^p,\quad s\in\R.
\end{equation}
Therefore, 
\[\boI(q)\geq C_1\|q\|_\Hilb^2-\varepsilon\|q\|_2^2- C_\varepsilon\|q\|_p^p,\quad \text{for all }q\in\Hilb_2\cap\boK.\]
Hence, the Sobolev embedding $\Hilb\subset \Ebanach$ yields
\[\boI(q)\geq (C_1-\varepsilon  C_2- C_3 C_\varepsilon\|q\|_\Hilb^{p-2})\|q\|_\Hilb^2,\quad \text{for all }q\in\Hilb_2\cap\boK,\]
for some constants $C_2,C_3>0$. Thus, taking $\varepsilon$, and then $r$, both small enough, the second inequality in \eqref{hyp:Splitting-PDE} follows.

To prove the first inequality in \eqref{hyp:Splitting-PDE}, let us define 
\[m_\eta =2\frac{1-\sqrt{1-\eta^2}}{\eta^2}, \text{ for all }\eta\in(0,1],\quad m_0=1.\]
It is clear that $\eta\mapsto m_\eta$ is a continuous increasing function in $[0,1]$. Then, 
\begin{equation}
    \Psi(q)\leq \dfrac{m_\eta}{2}\| q\|_\Hilb^2,\quad\hbox{for all }q\in\Hilb\hbox{ with } \|\nabla q\|_\infty\leq \eta.
\end{equation}
Taking into account that
\begin{equation}\label{eq:parseval}
q=\sum_{n=1}^k\beta_n\varphi_n,\quad\|q\|_2^2=\sum_{n=1}^k\beta_n^2,\quad \text{for all }q\in\Hilb_1,\text{ for some }\beta_1,\dots,\beta_k\in\R,
\end{equation}
standard elliptic estimates yield the existence of constants $C_4,C_5>0$ (depending on $a$ and $k$) such that
\begin{equation}\label{eq:elliptic_estimates}
        \|\nabla q\|_\infty\leq C_4\|q\|_\Hilb\leq C_4C_5\|q\|_2,\quad\hbox{for all }q\in\Hilb_1.
\end{equation}
Therefore, for any $\eta\in (0,1]$, we may take $r$ small enough so that
\begin{equation}\label{ineq:Psi-m_eta}
    \Psi(q)\leq \dfrac{m_\eta}{2}\| q\|_\Hilb^2,\quad\hbox{for all }q\in\Hilb_1\hbox{ with } \| q\|_\Hilb\leq r.
\end{equation}
Since the eigenfunctions are orthogonal in $L^2(\Omega)$, it is immediate that
\[\int_\Omega\Big(|\nabla q|^2+a(x)q^2\Big)dx=\sum_{n=1}^k \lambda_n \beta_n^2,\quad\text{for all }q=\sum_{n=1}^k\beta_n\varphi_n\in\Hilb_1.\]
As a consequence, \eqref{eq:negative_eigenvalues} and \eqref{eq:parseval} imply
\[\int_\Omega\Big(|\nabla q|^2+a(x)q^2\Big)dx\leq\lambda_k\|q\|_2^2,\quad\text{for all }q\in\Hilb_1.\]
Recalling by \eqref{eq:negative_eigenvalues} that $\lambda_k<0$, it follows from \eqref{eq:elliptic_estimates} that 
\[\int_\Omega\Big(|\nabla q|^2+a(x)q^2\Big)dx\leq\lambda_kC_5^{-1}\|q\|_\Hilb^2,\quad\text{for all }q\in\Hilb_1.\]
In sum, using \eqref{hyp:PDE-f1-multiplicity}, by the Sobolev embedding, we derive
\begin{align*}
\boI(q)\leq \int_\Omega\Big(\frac{m_\eta}{2}|\nabla q|^2 + \frac{a(x)}{2}q^2\Big)dx+\varepsilon\|q\|_2^2+ C_\varepsilon\|q\|_p^p\leq\frac12(m_\eta-1+C_5^{-1}\lambda_k+2\varepsilon C_2 + 2C_3C_\varepsilon\|q\|_\Hilb^{p-2})\|q\|_\Hilb^2,
\end{align*}
for all $q\in\Hilb_1$ with $\|q\|_\Hilb\leq r$. We conclude by taking $\eta$, $\varepsilon$ and $r$ (in this order) sufficiently small.
\end{proof}

\begin{proof}[Proof of Theorem~\ref{MainTheorem-Dirichlet}]
 By  Lemmas~\ref{lemma:functional-PDE},~\ref{lemma:PS-PDE},~\ref{lemma:FunctionalC2-Dirichlet},~\ref{lemma:splitting-PDE}, and by Remark \ref{remark:negative-minimum-PDE}, the hypotheses of Theorem~\ref{MainTheorem} are satisfied. As a consequence, $\boI$ admits two nonzero critical points. Finally, 
 Theorem~\ref{Th:criticalpoint-PDE} assures that such critical points are solutions to \eqref{eq:PDE-Minkowski}.
\end{proof}
		
{\bf Funding}. This work has been supported by the Spanish MCIU/AEI project PID2024–56079NA–I00. M.G. is partially supported by the Spanish MCIN/AEI project PID2022-136795NB-I00, project CEX2024-001517-M.

\appendix

\section*{Appendix A}
\makeatletter
\def\@currentlabel{A} 
\makeatother
\label{appendix-lemma}
\gdef\thelemma{A.\arabic{lemma}}
\setcounter{lemma}{0}
The proof will be performed for $\Hilb=H^1(\Omega)$, being analogous when the space $\Hilb$ is equipped with periodic or Dirichlet boundary conditions.
\begin{lemma}\label{appendix-lemma:tau-domain}
        The set $\boK$ defined in \eqref{eq:example_spaces} is a $\tau$-domain, where $\tau$ is induced by \eqref{eq:topology-K}. Moreover, $\boK\subset L^\infty(\Omega)$.
        \end{lemma}

        \begin{proof}
        One trivially has that $\boK$ is convex. Let us continue by proving that $\boK\subset L^\infty(\Omega)$. Indeed, let $q\in \boK$. Since $\|\nabla q\|_{\infty}\leq 1$, it follows that $q\in W^{1,p}(\Omega)$. If $p> N$, then the Sobolev embeddings yield $q\in L^\infty(\Omega)$. Otherwise, one has that $q\in W^{1,p_1}(\Omega)$, where $p_1=pN/(N-p)>p$ if $p<N$, and $p\ll p_1<\infty$ if $p=N$. Repeating the argument inductively, we arrive in a finite number of steps at $q\in L^\infty(\Omega)$.
        
        In order to show that $\boK$ is closed in $\Hilb$, we follow \cite[Lemma 2.2]{BerJebMaw}. Indeed, let $\{q_n\}\subset \Hilb$ be such that $\|q_n-q\|_\Hilb\to 0$ as $n\to\infty$ for some $q\in\Hilb$. In particular, $q_n\to q$ a.e. in $\Omega$. Hence, for every $x_0\in\Omega$ and $r_0>0$ such that the Euclidean ball $B(x_0,r_0)\subset\Omega$, the mean value theorem implies that
        \[|q_n(x)-q_n(y)|\leq \|\nabla q_n\|_\infty|x-y|\leq |x-y|,\quad\text{for all }x,y\in B(x_0,r_0).\]
        Passing to the pointwise limit, we derive that $q$ is locally Lipschitz, with Lipschitz constant $1$ at every ball contained in $\Omega$. A similar standard argument shows that $q\in L^\infty(\Omega)$ and, in particular, $q\in\boK$ (see e.g. \cite[Remark 4.2]{Heinonen}).

        Finally, let $\tau$ be the topology in $\boK$ induced by the norm \eqref{eq:topology-K}. Notice that $\tau$ is trivially stronger than the strong topology of $L^p(\Omega)$. Moreover, equipped with this topology, the balls \eqref{def:ball} are compact sets. Indeed, let $\{q_n\}\subset B_\boK(q_0,r)$. On the one hand, $\{q_n\}$ is bounded in $W^{1,p}(\Omega)$. Thus, there exists $q\in H^1(\Omega)$ such that, passing to a subsequence, $q_n\rightharpoonup q$ weakly in $H^1(\Omega)$ as $n\to\infty$. Therefore, $q_n\to q$ pointwise.
        On the other hand, since
        $\|\nabla q_n\|_{L^\infty(\Omega)}\leq 1$ for all $n$, there exists $V\in L^\infty(\Omega;\R^d)$ such that, passing to a subsequence, $\|\nabla q_n-V\|_{w^*}\to 0$ as $n\to\infty$. In particular, $\nabla q_n\rightharpoonup V$ weakly in $L^2(\Omega;\R^d)$, so $V=\nabla q$. Finally, since $W^{1,p}(\Omega)$ is compactly embedded in $L^p(\Omega)$, there exists $u\in L^p(\Omega)$ such that, passing again to a subsequence, $q_n\to u$ strongly in $L^p(\Omega)$ as $n\to\infty$. In particular, $q_n\to u$ pointwise, so $u=q$. 
    \end{proof}

\section*{Appendix B}
\makeatletter
\def\@currentlabel{B} 
\makeatother
\gdef\theexample{B.\arabic{example}}
\label{Appendix-Boundary}
\setcounter{example}{0}
We show that the positivity assumption on the boundary $\partial\boK$ is not compatible with the framework considered in this paper. For both equations \eqref{eq:LFE-intro} and \eqref{eq:PDE-Minkowski}, we provide a one-parameter family of functionals, each admitting a point on $\partial\boK$ at which $\boI$ takes a negative value.
\begin{example}
    With the notation of Section \ref{Sect:LFE}, let $A\in \boC^2_T$ satisfy \eqref{cond:A-decay} and \eqref{hyp:bound-secondderivatives}. Let also us assume that \eqref{def:equilibrium} and \eqref{def:isolated2} are both satisfied at $x_0=0$, and that $\nabla A(t,0)\equiv0$. In particular, $A(t,x)$ satisfies the hypotheses of Theorem \ref{MainTheorem:LFE}. Then there exists $q_0\in\boK\setminus\R^3$ such that $\boI(q_0)<0$. Thus, necessarily,
    \[\int_0^T\dq_0\cdot A(t,q_0)dt<0.\]
    Let us assume that $\|\dq_0\|_\infty<1$ and take $\alpha>1$. We denote by $\boK_\alpha$ the set analogous to $\boK$ for the period $(T/\alpha)$. Let us introduce the function 
    \[p(t)=\dfrac{q_0(\alpha t)}{\alpha\|\dq_0\|_\infty} ,\quad t\in\R,\]
   which is clearly $(T/\alpha)$-periodic with $\|\dot p\|_\infty=1$, hence $p\in\partial\boK_\alpha$. Let us also define the potential
    \[A_{\alpha}(t,x):=\alpha\|\dq_0\|_\infty A\left(\alpha t,\alpha\|\dq_0\|_\infty x\right).\]
    It is straightforward to verify that $A_\alpha\in\boC^2_{T/\alpha}$ and satisfies the assumptions of Theorem~\ref{MainTheorem:LFE} with period $(T/\alpha)$. Moreover, the associated action functional
    \begin{equation}
        \boI_{\alpha}(q):=\int_0^\frac{T}{\alpha}\big(1-\sqrt{1-|\dot q|^2}+\dot q\cdot A_{\alpha}(t,q)\big)dt,\quad q\in\boK_\alpha,
    \end{equation}
    fits in the variational framework of Section~\ref{Sect:LFE}. Then
    we deduce that
    \begin{align*}
       \boI_\alpha(p)\leq\frac{T}{\alpha}+\int_0^{\frac{T}{\alpha}}\dpp(t)\cdot A_{\alpha}(t,p(t))dt=\frac{T}{\alpha}+\int_0^T\dq_0(t)\cdot A(t,q_0(t))dt.
    \end{align*}
    Therefore, for every sufficiently large $\alpha>0$, there exists a function $p\in\partial\boK_\alpha$ such that $\boI_\alpha(p)<0$.
\end{example}

\begin{example}
Let $\Omega\subset\R^N$ be a bounded domain with smooth boundary. For every $\lambda\in\R$, let  $f_\lambda:\Omega\times\R\to\R$ be defined by $f_\lambda(x,s)=a(x)s-\lambda g(s),$ with $a\in L^\infty(\Omega)$ and $g\in\boC^1(\R)$ such that $g(0)=g'(0)=0$. It is clear that $f_\lambda$
satisfies the hypotheses of Theorem \ref{MainTheorem-Dirichlet} for every $\lambda\in\R$.

Assume for instance that $g(s)>0$ for all $s>0$, so that $G(s):=\int_0^sg(t)dt>0$ for every $s>0$. With the notation of Section \ref{Section:Minkowski},  take  $q_0\in\partial\boK$ such that $q_0>0$ in $\Omega$. The hypothesis on $G$ implies that $G(q)>0$ in $\Omega$. Therefore, denoting the associated action functional by $\boI_\lambda$, one may take $\lambda>0$ large enough so that $\boI_\lambda(q_0)<0$. 
\end{example}


\begin{thebibliography}{99}

    \bibitem{ALP} J. Aguirre, A. Luque, D. Peralta-Salas. Motion of charged particles in magnetic fields created by symmetric configurations of wires, \textsl{Phys. D}, {\bf 239} (2010), no. 10, 654--674.   
			
	\bibitem{ABT} D. Arcoya, C. Bereanu, P.J. Torres. Critical point theory for the Lorentz force equation, \textsl{Arch. Ration. Mech. Anal.}, {\bf 232} (2019), 1685--1724.
			
	\bibitem{ABT2} D. Arcoya, C. Bereanu, P.J. Torres. Lusternik–Schnirelmann theory for the action integral of the Lorentz force equation, \textsl{Calc. Var. Partial Differential Equations},  {\bf 59} (2020), no.~2, Paper No. 50, 32.
    
	\bibitem{AS}D. Arcoya, C. Sportelli. Relativistic equations with singular potentials. \textsl{Z. Angew. Math. Phys.} {\bf 74}, (2023), no.~3, Paper No. 91, 22 pp.
    
    \bibitem{BS}  R. Bartnik, L. Simon. Spacelike hypersurfaces with prescribed boundary values and mean curvature, \textsl{Comm. Math. Phys.} {\bf 87} (1982/83), no.~1, 131--152.
			
	\bibitem{Ber1} C. Bereanu. Interactions in the Lorentz force equation, \textsl{Calc. Var. Partial Differential Equations} {\bf 63} (2024), no.~2, Paper No. 29, 25 pp.
			
	\bibitem{Ber2} C. Bereanu. Mountain pass periodic solutions for the Lorentz force equation via the Poincar\'e action functional, \textsl{Commun. Contemp. Math.}  {\bf 27} (2025), no.~3, Paper No. 2450009, 15 pp.

	\bibitem{BerJebMaw} C. Bereanu, P. Jebelean, J. Mawhin. The Dirichlet problem with mean curvature operator in Minkowski space - a variational approach. \textsl{Adv. Nonlinear Stud,} {\bf 14} (2014), no.~2, 315--326.

    \bibitem{BerJebMaw-Neumann} C. Bereanu, P. Jebelean, J. Mawhin. Radial solutions of Neumann problems involving mean extrinsic curvature and periodic nonlinearities. \textsl{Calc. Var. Partial Differential Equations} {\bf 46} (2013), no.~1-2, 113--122.

    \bibitem{BerJebTor}  C. Bereanu, P. Jebelean, P.J. Torres. Multiple positive radial solutions for a Dirichlet problem involving the mean curvature operator in Minkowski space. \textsl{J. Funct. Anal.}  {\bf 265} (2013), no.~4, 644--659.
    
	\bibitem{Ber-Pir}C. Bereanu, A. Pirvuceanu. $S^1$-index theory for the Lorentz force equation. \textsl{Annali Scuola Normale Superiore - Classe di Scienze}, {\bf 25} (2026).		

    \bibitem{BonDavPom} D. Bonheure, P. d'Avenia, A. Pomponio. On the electrostatic Born--Infeld equation with extended charges. \textsl{Comm. Math. Phys.} {\bf 346} (2016), no. 3, 877--906.
    
	\bibitem{Bos} A. Boscaggin, F. Colasuonno, B. Noris, F. Sani. An Orlicz space approach to exponential elliptic problems in higher dimensions. \textsl{Calc. Var. Partial Differential Equations} {\bf 65} (2026), no.~3, Paper No. 85, 39 pp.
    
	\bibitem{BosDamPap} A. Boscaggin, W. Dambrosio, D. Pappini. Infinitely many periodic solutions to a Lorentz force equation with singular electromagnetic potential. \textsl{J. Differential Equations} {\bf 383} (2024), 190--213.
			
	\bibitem{Bre} H. Brezis. Functional analysis, Sobolev spaces and partial differential equations. \textsl{New York: Springer}, 2011.
	
	\bibitem{Bre-Nir} H. Brezis, L. Nirenberg. Remarks on finding critical points. \textsl{Comm. Pure Appl. Math.} {\bf 44} (1991), no.~8-9, 939--963.

    \bibitem{ByeIkoMalMar1} J. Byeon, N. Ikoma, A. Malchiodi, L. Mari. Compactness via monotonicity in non-smooth critical point theory, with application to Born--Infeld type equations. \textsl{J. Funct. Anal.} {\bf 290} (2026), no. 11, Paper No. 111438, 77 pp.

    \bibitem{ByeIkoMalMar2} J. Byeon, N. Ikoma, A. Malchiodi, L. Mari. Existence and regularity for prescribed Lorentzian mean curvature hypersurfaces, and the Born--Infeld model. \textsl{Ann. PDE} {\bf 10} (2024), no. 1, Paper No. 4, 86 pp.

    \bibitem{Chang} K.C. Chang. Variational methods for nondifferentiable functionals and their applications to partial differential equations. \textsl{J. Math. Anal. Appl.} {\bf 80} (1981), no.~1, 102--129.
			
	\bibitem{Deim} K. Deimling. Ordinary differential equations in Banach spaces. Lecture Notes in Math. Vol. 596, \textsl{Springer-Verlag, Berlin-New York}, 1977. vi+137 pp.
			
	\bibitem{Eke-Las} I. Ekeland, J.M. Lasry. On the number of periodic trajectories for a Hamiltonian flow on a convex energy surface.  \textsl{Ann. of Math.} (2) {\bf 112} (1980), no.~2, 283--319.
    
    \bibitem{FPV}A. Fiscella, A. Pinamonti, E. Vecchi. Multiplicity results for magnetic fractional problems. \textsl{J. Differential Equations} {\bf 263} (2017), no.~8, 4617--4633.
    
    \bibitem{KKP}D.~A. Kandilakis, N.~C. Kourogenis, N.~S. Papageorgiou. Two nontrivial critical points for nonsmooth functionals via local linking and applications. \textsl{ J. Global Optim.} {\bf 34} (2006), no.~2, 219--244 
    
	\bibitem{GLM} M. Garzón, S. López-Martínez. Action-minimizing periodic orbits of the Lorentz force equation with dominant vector potential. \textsl{Calc. Var. Partial Differential Equations} {\bf 65} (2026), no.~6, Paper No. 180, 20 pp.           

    \bibitem{GM} M. Garz\'on, S. Marò. Motions of a charged particle in the electromagnetic field induced by a non-stationary current. \textsl{Phys. D} {\bf 424} (2021), Paper No. 132945, 9 pp.
			
	\bibitem{GT} M. Garz\'on, P.J. Torres. Periodic solutions for the Lorentz force equation with singular potentials. \textsl{Nonlinear Anal. Real World Appl.} {\bf 56} (2020), 103162, 6 pp.
			
	\bibitem{GT2} M. Garz\'on, P.J. Torres. Periodic dynamics in the relativistic regime of an electromagnetic field induced by a time-dependent wire. \textsl{J. Differential Equations} {\bf 362} (2023), 173--197.
			
	\bibitem{GP1} F.G. Gasc\'on, D. Peralta-Salas. Motion of a charge in the magnetic field created by wires: impossibility of reaching the wires. \textsl{Phys. Lett. A} {\bf 333} (2004), no. 1-2, 72--78.
			
	\bibitem{GP2} F.G. Gasc\'on, D. Peralta-Salas. Some properties of the magnetic fields generated by symmetric configurations of wires. \textsl{Phys. D} {\bf 206} (2005), no. 1-2, 109--120.
	
    \bibitem{Heinonen} J. Heinonen.  Lectures on Lipschitz analysis. \textsl{Report. University of Jyv\"askyl\"a{} Department of Mathematics and Statistics,} {\bf 100}, Univ. Jyv\"askyl\"a, Jyv\"askyl\"a, 2005.

    \bibitem{Li-Liu} S. Li, J. Q. Liu. Some existence theorems on multiple critical points and their applications. \textsl{Kexue Tongbao}, 17 (1984), 1025–1027.
    
    \bibitem{Li-Liu2} S. Li, J. Q. Liu. Morse theory and asymptotic linear Hamiltonian system. \textsl{ J. Differential Equations} {\bf 78} (1989), no.~1, 53--73.

    \bibitem{Li-Willem} S. Li, M. Willem. Applications of local linking to critical-point theory. \textsl{ Journal of Mathematical Analysis and Applications}, 1995, vol. 189, no 1.

    \bibitem{LMM}R. Livrea, S.~A. Marano, D. Motreanu. Critical points for nondifferentiable functions in presence of splitting. \textsl{J. Differential Equations} {\bf 226} (2006), no.~2, 704--725.

    \bibitem{MBR} G. Molica~Bisci, V.~D. R\u adulescu. Applications of local linking to nonlocal Neumann problems. \textsl{Commun. Contemp. Math.} {\bf 17} (2015), no.~1, 1450001, 17 pp.
    
    \bibitem{Perera} K. Perera. Homological local linking. \textsl{Abstr. Appl. Anal.} {\bf 3} (1998), no.~1-2, 181--189
			
	\bibitem{Po2} H. Poincar\'e. Sur la dynamique de l'\'electron. \textsl{Rend. Circ. Mat. Palermo} {\bf 21} (1906), 129--176.
			
	\bibitem{Szu} A. Szulkin. Minimax principles for lower semicontinuous functions and applications to nonlinear boundary value problems. \textsl{Ann. Inst. H. Poincar\'e Anal. Non Lin\'eaire} {\bf 3} (1986), no.2, 77--109.
	    
	\bibitem{Willem} M. Willem. Minimax Theorems. Progr. Nonlinear Differential Equations Appl., 24, \textsl{Birkh\"auser Boston, Inc., Boston, MA,} 1996. x+162 pp. 

    \bibitem{Wu} X. Wu. A new critical point theorem for locally Lipschitz functionals with applications to differential equations. \textsl{Nonlinear Anal.} {\bf 66} (2007), no.~3, 624--638.
			
			
\end{thebibliography}
\end{document}